\documentclass[12pt]{amsart}
\usepackage{amssymb}
\usepackage{amsfonts}
\usepackage{indentfirst}
\usepackage{amsopn}
\usepackage{amsmath}
\usepackage{amsthm}
\usepackage{amscd}
\usepackage{graphicx}
\usepackage{epstopdf}

\usepackage[english]{babel}
\usepackage{amsmath}
\usepackage{graphicx}
\usepackage{hyperref}
\usepackage{enumerate}

\newtheorem{lemma}{Lemma}[section]
\newtheorem{theorem}[lemma]{Theorem}
\newtheorem{proposition}[lemma]{Proposition}
\newtheorem{corollary}[lemma]{Corollary}

\newtheorem{remark}[lemma]{Remark}

\numberwithin{equation}{section}

\newcommand{\beq}{\begin{equation}}
\newcommand{\eeq}{\end{equation}}
\newcommand{\be}{\begin{equation*}}
\newcommand{\ee}{\end{equation*}}

\newcommand{\RE}{\mathbb R}

\newcommand{\NA}{\mathbb N}

\usepackage{latexsym,epsfig,bm,xcolor}

\begin{document}

\title[Standing waves of the quintic NLS equation]{\bf Standing waves of the quintic NLS equation \\ on the tadpole graph}

\author[D. Noja]{Diego Noja}
\address[D. Noja]{Dipartimento di Matematica e Applicazioni, Universit\`a di Milano Bicocca, via R. Cozzi 55, 20126 Milano, Italy}
\email{diego.noja@unimib.it}

\author[D.E. Pelinovsky]{Dmitry E. Pelinovsky}
\address[D.E. Pelinovsky]{Department of Mathematics and Statistics, McMaster University,
Hamilton, Ontario, Canada, L8S 4K1}
\email{dmpeli@math.mcmaster.ca}

\date{\today}
\maketitle

\begin{abstract}
The tadpole graph consists of a circle and a half-line attached at a vertex. We analyze standing waves of the
nonlinear Schr\"{o}dinger equation with quintic power nonlinearity equipped with the Neumann-Kirchhoff boundary conditions at the vertex.
The profile of the standing wave with the frequency $\omega\in (-\infty,0)$ is characterized as a
global minimizer of the quadratic part of energy constrained to the unit sphere in $L^6$.
The set of standing waves includes the set of ground states, which are the global minimizers of the energy
at constant mass ($L^2$-norm), but it is actually wider. While ground states exist only for a certain interval of masses,
the standing waves exist for every $\omega \in (-\infty,0)$ and correspond to a bigger interval of masses. It is proven
that there exist critical frequencies $\omega_1$ and $\omega_0$ with $-\infty < \omega_1 < \omega_0 < 0$
such that the standing waves are the ground state for $\omega \in [\omega_0,0)$,
local constrained minima of the energy for $\omega \in (\omega_1,\omega_0)$ and
saddle points of the energy at constant mass for $\omega \in (-\infty,\omega_1)$.
Proofs make use of the variational methods and the analytical theory for differential equations.
\end{abstract}
\vskip10pt
\begin{footnotesize}

 \end{footnotesize}

\section{Introduction}

The analysis of nonlinear PDEs on metric graphs has recently attracted a certain attention \cite{NojaBook}.
One of the reason is potential applicability of this analysis to physical models such as
Bose-Einstein condensates trapped in narrow potentials with T-junctions or X-junctions,
or networks of optical fibers. Another reason is the possibility to rigorously prove
a complicated behavior of the standing waves due to the interplay between geometry
and nonlinearity, which is hardly accessible in higher dimensional problems.

The most studied nonlinear PDE on a metric graph $\mathcal G$ is the nonlinear Schr\"{o}dinger (NLS)
equation with power nonlinearity, which we take in the following form:
\begin{equation}\label{tNLS}
i \frac d {dt} \Psi =  \Delta \Psi + (p+1)| \Psi |^{2p} \Psi,
\end{equation}
where the wave function $\Psi(t,\cdot)$ is defined componentwise on edges of the graph $\mathcal G $ subject to
suitable boundary conditions at vertices of the graph $\mathcal G $. The Laplace operator $\Delta$ and
the power nonlinearity are also defined componentwise. The natural Neumann--Kirchhoff boundary conditions
are typically added at the vertices to ensure that $\Delta$ is self-adjoint in $L^2(\mathcal G)$
with a dense domain $D(\Delta) \subset L^2(\mathcal G)$ \cite{BK13,Exner}.

The Cauchy problem for the NLS equation (\ref{tNLS}) is locally well-posed in the energy space
$H^1_{\rm C}(\mathcal G) := H^1(\mathcal G) \cap C^0(\mathcal G)$, which is the space of
the componentwise $H^1$ functions that are continuous across the vertices of the graph $\mathcal G$
\cite{CFN17,GO19,KP2}. The following two conserved quantities
of the NLS equation (\ref{tNLS}) are defined in $H^1_{\rm C}(\mathcal G)$, namely {\it the mass}
\beq
\label{mass}
Q(\Psi) := \| \Psi \|_{L^2(\mathcal G)}^2
\eeq
and {\it the total energy}
\beq
E(\Psi) = \| \nabla \Psi \|_{L^2(\mathcal G)}^2 - \|\Psi \|_{L^{2p+2}(\mathcal G)}^{2p+2}.
\eeq
Due to conservation of mass and energy and the Gagliardo--Nirenberg inequality (for which see \cite{AST15,AST16}),
unique local solutions to the NLS equation (\ref{tNLS})
in $H^1_{\rm C}(\mathcal G)$ are extended globally in time for subcritical ($p<2$) nonlinearities
and for the critical ($p=2$) nonlinearity in the case of small initial data.

Standing waves of the NLS equation \eqref{tNLS} are solutions of the form
$\Psi(t,x) = e^{i\omega t} \Phi(x)$, where $\Phi$ satisfies the elliptic system
\beq \label{eqstazp}
-\Delta \Phi - (p+1) |\Phi|^{2p} \Phi = \omega \Phi
\eeq
and $\omega \in \RE$ is a real parameter. We refer to $\omega$ as the frequency of the standing wave and
to $\Phi$ as to the spatial profile of the standing wave. The stationary NLS equation (\ref{eqstazp})
is the Euler--Lagrange equation for the augmented energy functional or simply \emph{the action}
\beq
\label{action}
S_{\omega}(U) := E(U) - \omega Q(U),
\eeq
which is defined for every $U \in H^1_{\rm C}(\mathcal G)$. If the infimum of the constrained minimization problem:
\begin{equation}
\label{ground-state}
\mathcal{E}_{\mu} = \inf_{U \in H^1_{\rm C}(\mathcal{G})} \left\{ E(U) : \;\; Q(U) = \mu \right\},
\end{equation}
is finite and is attained at $\Phi \in H^1_{\rm C}(\mathcal{G})$ so that $\mathcal{E}_{\mu} = E(\Phi)$
and $\mu = Q(\Phi)$, we say that this $\Phi$ is {\em the ground state}. By the usual bootstrapping arguments, the same $\Phi$ is also
a strong solution $\Phi \in H^2_{\rm NK}(\mathcal{G})$ to the stationary NLS equation (\ref{eqstazp}) with the
corresponding Lagrange multiplier $\omega$ which depends on the mass $\mu$. Here $H^2_{\rm NK}(\mathcal{G})$ is the space of
the componentwise $H^2$ functions that satisfy the natural Neumann--Kirchhoff boundary conditions
across the vertices of the graph $\mathcal G$. This space coincides with the
domain $D(\Delta)$ of the Laplace operator $\Delta$.

Ground states on metric graphs with Neumann--Kirchhoff boundary conditions at the vertices of $\mathcal{G}$
and no external potentials exist under rather restrictive topological conditions (see \cite{AST15,AST16,AST17,AST19}).
When delta-impurities at the vertices or external potentials give rise to a negative eigenvalue of
the linearized operator at the zero solution, a ground state always exist in the subcritical \cite{CFN17},
critical \cite{C18}, and supercritical \cite{Ardila18} cases.

The aim of this paper is to characterize the standing waves for the $L^2$-critical (quintic, $p=2$) NLS equation
in the particular case of the tadpole graph $\mathcal{T}$.
The tadpole graph $\mathcal{T}$ is the metric graph $\mathcal{G}$ constituted by a circle
and a half-line attached at a single vertex.
We normalize the interval for the circle to $[-\pi,\pi]$
with the end points connected to the half-line $[0,\infty)$ at a single vertex. The natural Neumann--Kirchhoff
boundary conditions for the two-component vectors $U := (u,v) \in H^2(-\pi,\pi)\times H^2(0,\infty)$ are given by
\beq
\label{dd}
\left\{ \begin{aligned}
&  u(\pi) = u(-\pi) = v(0), \\
& u'(\pi) - u'(-\pi) = v'(0).
\end{aligned} \right.
\eeq
The Laplace operator $\Delta : H^2_{\rm NK}(\mathcal{T}) \subset L^2(\mathcal{T}) \mapsto L^2(\mathcal{T})$
with the operator domain
\beq\label{DomainDelta}
H^2_{\rm NK}(\mathcal{T}) := \left\{ u \in H^2(-\pi,\pi), \;\; v \in H^2(0,\infty) : \;\; \mbox{\rm satisfying} \; (\ref{dd}) \right\}
\eeq
is self-adjoint in $L^2(\mathcal{T}) := L^2(-\pi,\pi) \times L^2(0,\infty)$.
Integrating by parts yields for every $U = (u,v) \in H^2_{\rm NK}(\mathcal{T})$:
\beq
\label{by-parts}
\langle -\Delta U, U \rangle = \| \nabla U \|^2_{L^2(\mathcal{T})} + v'(0) v(0) - u'(\pi) u(\pi) + u'(-\pi) u(-\pi) =
\| \nabla U \|^2_{L^2(\mathcal{T})} \geq 0,
\eeq
which implies that $\sigma(-\Delta) \subseteq [0, \infty)$. Appendix \ref{app-spectrum} gives the precise characterization
of $\sigma(-\Delta) = [0,\infty)$ which includes the absolute continuous part of the spectrum denoted by $\sigma_{\rm ac}(-\Delta)$ and a countable set of
embedded eigenvalues.

The tadpole graph $\mathcal{T}$ has been proven to be a good testing ground for a more general study.
A first classification of standing waves for the cubic ($p=1$) NLS on the tadpole graph
was given in \cite{cfn15}, then it was extended to the subcritical case $p\in (0,2 )$ in \cite{[NPS15]} where
orbital stability of some standing waves has been considered.

By Theorem 2.2 in \cite{AST15} for the subcritical case $p \in (0,2)$,
$\mathcal{E}_{\mu}$ in (\ref{ground-state}) satisfies the bounds
\begin{equation}
\label{bounds-on-E}
\mathcal{E}_{\mathbb{R^+}} \leq \mathcal{E}_{\mu} \leq \mathcal{E}_{\mathbb{R}},
\end{equation}
where $\mathcal{E}_{\mathbb{R^+}}$ is the energy of a half-soliton of the NLS equation on a half-line with the same mass $\mu$
and $\mathcal{E}_{\mathbb{R}}$ is the energy of a full soliton on a full line with the same mass $\mu$. By Theorem
3.3 and Corollary 3.4 in \cite{AST16}, the infimum is attained if there exists $\Psi_* \in H^2_{\rm NK}(\mathcal{G})$
such that $E(\Psi_*) \leq \mathcal{E}_{\mathbb{R}}$. Based on this criterion, it was shown in \cite{AST16} that
the subcritical NLS equation for the tadpole graph $\mathcal{T}$ admits the ground state $\Phi$ {\it for all
positive values of the mass $\mu$}.
Moreover, by using suitable symmetric rearrangements it was shown in \cite{AST16}
that the ground state $\Phi$ is given by a monotone piece of soliton on the half-line glued
with a piece of a periodic function on the circle, with a single maximum sitting at the antipodal point to the vertex
(see also \cite{cfn15,[NPS15]}).

In the critical power $p = 2$, it was shown in \cite[Theorem 3.3]{AST17} that the ground state
on the metric graph $\mathcal{G}$ with exactly one half-line (e.g., on the tadpole graph $\mathcal{T}$)
is attained if and only if $\mu \in (\mu_{\mathbb{R}^+},\mu_{\mathbb{R}}]$, where
$\mu_{\mathbb{R}^+}$ is the mass of the half-soliton of the NLS equation on the half-line
and $\mu_{\mathbb{R}}$ is the mass of the full-soliton on the full line,
both values are independent on $\omega$ for $p = 2$.
Indeed, let $\varphi_{\omega}(x) = |\omega|^{1/4} {\rm sech}^{1/2}(2 \sqrt{|\omega|} x)$ be
a soliton of the quintic NLS equation on the line centered at $x = 0$, then we compute
\begin{equation}
\label{mass-half-soliton}
\mu_{\RE^+} = \| \varphi_{\omega} \|_{L^2(\RE^+)}^2 = \frac{\pi}{4}
\end{equation}
and
\begin{equation}
\label{mass-soliton}
\mu_{\RE}  = \| \varphi_{\omega} \|_{L^2(\RE)}^2 =\frac{\pi}{2}.
\end{equation}
Thus, the ground state on the tadpole graph $\mathcal{T}$ exists if and only if
$\mu \in (\mu_{\mathbb{R}^+},\mu_{\mathbb{R}}]$; moreover, $\mathcal{E}_{\mu} < 0$.
It was also shown in \cite[Proposition 2.4]{AST17} that $\mathcal{E}_{\mu} = 0$ if
$\mu \leq \mu_{\mathbb{R}^+}$ and $\mathcal{E}_{\mu} = -\infty$ if $\mu > \mu_{\mathbb{R}}$.
We conjecture that this behavior of $\mathcal{E}_{\mu}$ for the critical NLS equation (\ref{tNLS}) with $p = 2$
is associated to the decay of strong solutions $\Psi(t,\cdot) \in H^2_{\rm NK}(\mathcal{T})$
to zero as $t \to \infty$ if the mass $\mu$ of the initial data $\Psi(0,\cdot) = \Psi_0$
satisfies $\mu \leq \mu_{\mathbb{R}^+}$ and the blow-up in a finite time $t$ if $\mu > \mu_{\mathbb{R}}$.
The latter behavior is known on the full line \cite{Cazenave} but it has not been proven yet
in the context of the unbounded metric graph $\mathcal{T}$ (strong instability
of bound states on star graphs was recently analyzed in \cite{GO19}).

The main novelty of this paper is to explore the variational methods and the analytical theory
for differential equations in order to construct the standing waves with profile $\Phi$ satisfying
the elliptic system (\ref{eqstazp}) with $p = 2$, rewritten again as
\beq \label{eqstaz}
-\Delta \Phi - 3 \Phi^5 = \omega \Phi.
\eeq
The variational construction relies on the following constrained minimization problem:
\begin{equation}
\label{infB}
\mathcal{B}(\omega) = \inf_{U \in H^1_{\rm C}(\mathcal{T})} \left\{ B_{\omega}(U) : \quad \| U \|_{L^{6}(\mathcal{T})} = 1 \right\}, \quad \omega < 0,
\end{equation}
where
\begin{equation}\label{Bfunctional}
B_{\omega}(U) := \| \nabla U \|_{L^2(\mathcal{T})}^2 - \omega \| U \|^2_{L^2(\mathcal{T})}.
\end{equation}
We are not aware of previous applications of the variational problem (\ref{infB})
in the context of the NLS equation on metric graphs. The variational problem (\ref{infB}) gives generally a larger set of standing waves
compared to the set of ground states in the variational problem (\ref{ground-state}). This
is relevant for the orbital stability of the standing waves.

Versions of the variational problem (\ref{infB}) arise in the determination of
the best constant of the Sobolev inequality, which is equivalent to the Gagliardo--Nirenberg inequality
in $\RE^n$ (see, for example, \cite{Agueh06, Agueh08, DELL14, MorPizz18} and references therein).
However, as follows from \cite{AST17} and it is shown in Lemma \ref{lemma-1-4} below,
the minimizer of (\ref{infB}) does not give the best constant in the Gagliardo--Nirenberg inequality
on the tadpole graph $\mathcal{T}$.

Another well-known variational problem is the minimization of the action functional (\ref{action})
on the Nehari manifold. This approach was used in \cite{FOO2008} for the so-called delta potential
on the line and generalized in \cite{[ACFN14]} in the context of a star graph with a delta potential
 at the vertex. More recently, the variational problems at the Nehari manifolds were analyzed in \cite{AkPankov19,Pankov18}.
In Appendix \ref{app-relation}, we show how the constrained minimization problem (\ref{infB}) is related to the minimization of
the action  (\ref{action}) on the Nehari manifold defined by the constraint $B_{\omega}(U) = 3 \| U \|^6_{L^6(\mathcal{T})}$.

We shall now present the main results of this paper. The first theorem
states that the variational problem (\ref{infB}) determines
a family of standing waves $\Phi(\cdot,\omega)$ to the elliptic system (\ref{eqstaz}) for every $\omega < 0$.

\begin{theorem}
For every $\omega < 0$, there exists a global minimizer $\Psi(\cdot,\omega) \in H^1_{\rm C}(\mathcal{T})$
of the constrained minimization problem (\ref{infB}), which yields a strong solution
$\Phi(\cdot,\omega) \in H^2_{\rm NK}(\mathcal{T})$ to the stationary NLS equation (\ref{eqstaz}).
The standing wave $\Phi$ is real up to the phase rotation, positive up to the sign choice, symmetric
on $[-\pi,\pi]$ and monotonically decreasing on $[0,\pi]$ and $[0,\infty)$.
\label{theorem-existence}
\end{theorem}

The main idea in the proof of Theorem \ref{theorem-existence} is a compactness argument
which eliminates the possibility that the minimizing sequence splits or escapes to infinity along the unbounded edge of the
tadpole graph $\mathcal{T}$.

In what follows, we usually omit the dependence on $\omega$ for $\Psi(\cdot,\omega)$ and $\Phi(\cdot,\omega)$.
The linearization of the stationary NLS equation (\ref{eqstaz}) around $\Phi$ is defined by the self-adjoint operator
$\mathcal{L} : H^2_{\rm NK}(\mathcal{T}) \subset L^2(\mathcal{T}) \mapsto L^2(\mathcal{T})$ given by the
following differential expression:
\beq \label{Jacobian}
\mathcal{L} = -\Delta - \omega - 15 \Phi^{4}.
\eeq
Since it is self-adjoint, the spectrum of $\mathcal{L}$ in $L^2(\mathcal{T})$ is a subset of real line.
Since $\Phi(x) \to 0$ as $x \to \infty$ exponentially on the half-line, application of Weyl's Theorem yields that
the absolutely continuous spectrum of $\mathcal{L}$ is given by
\beq
\sigma_{\rm a.c.}(\mathcal{L}) = \sigma(-\Delta - \omega) = [|\omega|,\infty),
\eeq
and that there are only finitely many eigenvalues of $\mathcal{L}$ located below $|\omega|$
with each eigenvalue having finite multiplicity.

Let $n(\mathcal{L})$ be the Morse index (the number of negative
eigenvalues of $\mathcal{L}$ with the account of their multiplicities) and $z(\mathcal{L})$ be the nullity
index of $\mathcal{L}$ (the multiplicity of the zero eigenvalue of $\mathcal{L}$). Since
\beq
\label{negative-result}
\langle \mathcal{L} \Phi, \Phi \rangle_{L^2(\Gamma_N)} = - 12 \| \Phi \|_{L^6(\mathcal{T})}^{6} < 0,
\eeq
there is always a negative eigenvalue of $\mathcal{L}$ so that $n(\mathcal{L}) \geq 1$.
Since $\Phi$ is obtained from the variational problem (\ref{infB}) with only one constraint, by Courant's Min-Max Theorem,
we have $n(\mathcal{L}) \leq 1$, hence $n(\mathcal{L}) = 1$. In addition, we prove that the operator $\mathcal{L}$ is non-degenerate for every $\omega < 0$ with $z(\mathcal{L}) = 0$.
These facts are collected together in the following theorem.

\begin{theorem}
Let $\Phi \in H^2_{\rm NK}(\mathcal{T})$ be a solution to the stationary NLS equation (\ref{eqstaz})
for $\omega < 0$ constructed in Theorem \ref{theorem-existence}. Then,
$n(\mathcal{L}) = 1$ and $z(\mathcal{L}) = 0$ for every $\omega < 0$.
\label{theorem-degeneracy}
\end{theorem}

The proof of Theorem \ref{theorem-degeneracy} relies on the dynamical system methods and the
analytical theory for differential equations. In particular, we construct the standing wave of Theorem \ref{theorem-existence} by
using orbits of a conservative system on a phase plane and by introducing the period function, whose analytical properties are useful to
prove monotonicity of parametrization of the standing wave in Lemma \ref{lemma-2-1}
and the non-degeneracy of the linearized operator $\mathcal{L}$ in Lemma \ref{lemma-3-1}.

It follows from the non-degeneracy of $\mathcal{L}$ that the map
$(-\infty,0) \ni \omega \mapsto \Phi(\cdot,\omega) \in H^2_{\rm NK}(\mathcal{T})$
is $C^1$. The following theorem presents the monotonicity properties of the mass $\mu(\omega) := Q(\Phi(\cdot,\omega))$
as a function of $\omega$ needed for analysis of orbital stability of the standing waves with profile $\Phi$.

\begin{theorem}
Let $\Phi(\cdot,\omega) \in H^2_{\rm NK}(\mathcal{T})$ be the solution to the stationary NLS equation (\ref{eqstaz})
for $\omega < 0$ constructed in Theorem \ref{theorem-existence}. Then,
the mapping $\omega \mapsto \mu(\omega)= Q(\Phi(\cdot,\omega))$ is $C^1$ for every $\omega < 0$ and satisfies
\begin{equation}
\label{mass-property-1}
\mu(\omega) \to \mu_{\RE^+} \;\; \mbox{\rm as} \;\; \omega \to 0
\quad \mbox{\rm and} \quad \mu(\omega) \to \mu_{\RE} \;\; \mbox{\rm as} \;\; \omega \to -\infty.
\end{equation}
Moreover, there exist $\omega_1$ and $\omega_0$ satisfying $-\infty < \omega_1 < \omega_0 < 0$
such that
\begin{equation}
\label{mass-property-2}
\mu'(\omega) > 0 \;\; \mbox{\rm for} \;\; \omega \in (-\infty,\omega_1) \quad \mbox{\rm and} \quad
\mu'(\omega) < 0 \;\; \mbox{\rm for} \;\; \omega \in (\omega_1,0)
\end{equation}
and
\begin{equation}
\label{mass-property-3}
\mu(\omega) \notin (\mu_{\RE^+},\mu_{\RE}] \;\; \mbox{\rm for} \;\; \omega \in (-\infty,\omega_0) \quad \mbox{\rm and} \quad
\mu(\omega) \in (\mu_{\RE^+},\mu_{\RE}] \;\; \mbox{\rm for} \;\; \omega \in [\omega_0,0).
\end{equation}
\label{theorem-persistence}
\end{theorem}

The proof of the asymptotic limits (\ref{mass-property-1}) in Theorem \ref{theorem-persistence}
relies on the asymptotic methods involving power series expansions and properties of Jacobian elliptic functions.
The proof of monotonicity (\ref{mass-property-2}) is performed with the analytical theory for differential equations.
The final property (\ref{mass-property-3}) follows from (\ref{mass-property-1}) and (\ref{mass-property-2}).

Since $n(\mathcal{L}) = 1$ and $z(\mathcal{L}) = 0$ by Theorem \ref{theorem-degeneracy},
the following Corollary \ref{corollary-persistence} follows from Theorem \ref{theorem-persistence}
by the orbital stability theory of standing waves (see the recent application of this theory on star graphs in \cite{KP1,KP2,KGP}).

\begin{corollary}
\label{corollary-persistence}
The standing wave $\Phi(\cdot,\omega)$ for $\omega \in (\omega_1,0)$ is
a local constrained minimizer of the energy $E(U)$ subject to the constraint $Q(U) = \mu(\omega)$,
whereas for $\omega \in (-\infty,\omega_1)$, it is a saddle point of the energy $E(U)$ subject
to the constraint $Q(U) = \mu(\omega)$,
\end{corollary}

\begin{remark}\label{uniquenessGS}
It follows from the dynamical system methods in the proof of Theorem \ref{theorem-degeneracy}
that there exists the unique solution of the stationary NLS equation \eqref{eqstaz} (up to the phase rotation)
with the properties stated in Theorem \ref{theorem-existence}. Therefore, for every $\omega \in [\omega_0,0)$ such that
$\mu(\omega)\in (\mu_{\RE^+}, \mu_\RE]$, the minimizers of the variational problem \eqref{infB} coincides
with the ground state of the variational problem \eqref{ground-state},
which shares the same properties (see \cite{AST17} and \cite{DST20}).
\end{remark}

Fig. \ref{fig-mass} shows the mapping $\omega \mapsto \mu(\omega)$ obtained by using numerical approximations.
This numerical result agrees with the statement of Theorem \ref{theorem-persistence}.

\begin{figure}[h!]
	\centering
	\includegraphics[width=12cm,height=8cm]{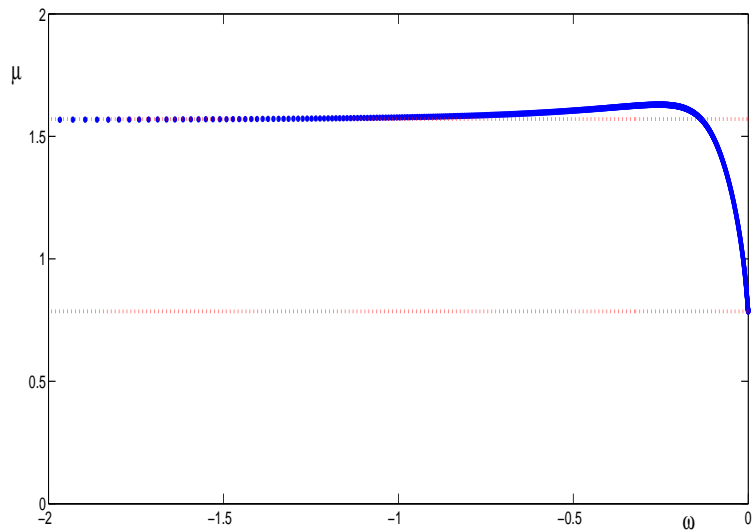}
	\caption{Mass $\mu$ versus frequency $\omega$ for the minimizer of
the constrained minimization problem (\ref{infB}). The horizontal dotted lines show the limiting levels
(\ref{mass-half-soliton}) and (\ref{mass-soliton}) given by the half-soliton
mass $\mu_{\RE^+}$ and the full-soliton mass $\mu_{\RE}$.}
	\label{fig-mass}
\end{figure}

We summarize that the standing wave
$\Phi(\cdot,\omega)$ is the ground state (global minimizer) of the variational problem (\ref{ground-state}) for $\omega \in [\omega_0,0)$,
a local constrained minimizer for $\omega \in (\omega_1,\omega_0)$, and a saddle point of the energy $E(U)$ subject
to the constraint $Q(U) = \mu(\omega)$ for $\omega \in (-\infty,\omega_1)$. We stress
that both the intervals $(-\infty,\omega_1)$ and $(\omega_1,\omega_0)$ correspond to $\mu(\omega) > \mu_{\RE}$.
As a result, although no ground state defined by the variational problem (\ref{ground-state})
exists for $\mu > \mu_{\RE}$ as a consequence of Theorem 3.3 in \cite{AST17}, there exists
a local constrained minimizer of energy for fixed mass $\mu \in (\mu_{\RE},\mu_{\rm max})$,
where $\mu_{\rm max} = \mu(\omega_1)$ is the maximal value of the mapping $\omega \mapsto \mu(\omega)$.

In connection with the variational characterization of the standing waves on metric graphs $\mathcal{T}$ which are not necessarily
the ground states, we mention two recent papers treating situations different than ours.
In \cite{PSV19}, local and not global constrained minimizers of the energy
at constant mass in the critical power $p = 2$ are discussed for some cases of unbounded graphs
with Neumann--Kirchhoff boundary conditions. In \cite{AST19}, local minima of the energy at constant mass
for subcritical power are constructed for general graphs by means of a variational  problem with two constraints.

The present paper is organized as follows. Section \ref{sec-variational} gives the proof of Theorem \ref{theorem-existence}
by using the variational characterization of the standing waves.
Section \ref{sec-dynamical} gives the proof of Theorem \ref{theorem-degeneracy}
by using the dynamical system methods and the analytical theory for differential equations.
Section \ref{sec-mass} gives the proof of Theorem \ref{theorem-persistence}.
Appendix \ref{app-spectrum} gives the precise characterization of the spectrum of the Laplace operator $\Delta$
on the tadpole graph $\mathcal{T}$.
Appendix \ref{app-relation} gives information between the variational problem
(\ref{infB}) and the minimization of the action (\ref{action}) at the Nehari manifold.
Appendix \ref{app-asymptotics} gives computational details
of approximating of the integral for the mass $\mu(\omega)$ in the limit $\omega \to -\infty$.

\section{Variational characterization of the standing waves}
\label{sec-variational}

Here we shall prove Theorem \ref{theorem-existence}.
First, we show that there exists a global minimizer $\Psi \in H^1_{\rm C}(\mathcal{T})$
of the variational problem (\ref{infB}) for every $\omega < 0$.
Then, we deduce properties of the minimizer and use the Lagrange multipliers
to obtain the solution $\Phi \in H^2_{\rm NK}(\mathcal{T})$ to the stationary NLS equation (\ref{eqstaz}).

We begin by recalling that the variational problem \eqref{infB} has 
a solution on both $\RE$ and $\RE^+$ (see for example the already mentioned papers \cite{Agueh06, Agueh08} or references therein). 
The precise value of the infima $\mathcal{B}_{\mathbb{R}}(\omega)$ and  $\mathcal{B}_{\mathbb{R}^+}(\omega)$ are given 
in the subsequent formula \eqref{inferreerreplus}. Let us now consider the tadpole graph. 

It follows from (\ref{Bfunctional}) that 
$B_{\omega}(U)$ with $\omega < 0$ is equivalent to $\| U \|_{H^1(\mathcal{T})}^2$ in the sense
that there exist positive constants $C_{\pm}(\omega)$ such that for every $U \in H^1_C(\mathcal{T})$, it is true that
\begin{equation}
\label{B-equivalence}
C_-(\omega) \| U \|^2_{H^1(\mathcal{T})} \leq B_{\omega}(U) \leq C_+(\omega) \| U \|^2_{H^1(\mathcal{T})}.
\end{equation}
Hence, it follows that $B_{\omega}(U) \geq 0$ so that the infimum $\mathcal{B}(\omega) > 0$ of the
variational problem (\ref{infB}) exists. Positivity of $\mathcal{B}(\omega)$
follows from the nonzero constraint $\| U \|_{L^6(\mathcal{T})} = 1$ and Sobolev's embedding
of $H^1(\mathcal{T})$ to $L^6(\mathcal{T})$:
there exists a $U$-independent constant $C > 0$ such that
\begin{equation}
\label{Sobolev-embedding}
\| U \|_{L^6(\mathcal{T})} \leq C \| U \|_{H^1(\mathcal{T})}
\end{equation}
for all $U \in H^1_{\rm C}(\mathcal{T})$.

Let $\{ U_n \}_{n \in \mathbb{N}}$ be a minimizing sequence in $H^1(\mathcal{T})$ such that
$\| U_n \|_{L^6(\mathcal{T})} = 1$ for every $n \in \mathbb{N}$ and
$B_{\omega}(U_n) \to \mathcal{B}(\omega)$ as $n \to \infty$.
Therefore, there exists a weak limit of the sequence in $H^1(\mathcal{T})$
denoted by $U_*$ so that
\begin{equation}
\label{weak-convergence}
U_n \rightharpoonup U_* \quad \mbox{in} \ H^1(\mathcal{T}),  \quad  \mbox{as} \quad n\rightarrow \infty.
\end{equation}
By Fatou's Lemma, we have
$$
0 \leq \gamma := \| U_* \|^6_{L^6(\mathcal{T})} \leq \lim_{n \to \infty} \| U_n \|^6_{L^6(\mathcal{T})} = 1,
$$
so that $\gamma \in [0,1]$. The following two lemmas eliminate the cases $\gamma \in (0,1)$ and $\gamma = 0$.

\begin{lemma}
\label{lemma-1-0}
For every $\omega < 0$, either $\gamma = 0$ or $\gamma = 1$. If $\gamma = 1$, then
$U_* \in H^1(\mathcal{T})$ is a global minimizer of (\ref{infB}).
\end{lemma}

\begin{proof}
It follows from (\ref{B-equivalence}) and (\ref{weak-convergence}) that
$$
B_{\omega}(U_n) = B_{\omega}(U_*-U_n) + B_{\omega}(U_*) + {\rm o}(1),
$$
where ${\rm o}(1) \to 0$ as $n \to \infty$. As a result, we obtain
\begin{equation}\label{quadratic}
\mathcal B(\omega)=\lim _{n\to \infty}B_{\omega}(U_n) = B_{\omega}(U_*) + \lim_{n\to \infty} B_{\omega}(U_n-U_*).
\end{equation}
It follows from the Brezis-Lieb Lemma (see \cite{BL}) that
\begin{equation}\label{brezis-lieb}
1= \lim_{n \to \infty} \| U_n \|^6_{L^6(\mathcal{T})}=\| U_* \|^6_{L^6(\mathcal{T})}+\lim_{n\to \infty}\| U_n-U_* \|^6_{L^6(\mathcal{T})}\ .
\end{equation}
It follows from \eqref{quadratic} after normalizing in $L^6(\mathcal{T})$ the arguments of
$U_*$ and $U_n-U_*$ and taking into account \eqref{brezis-lieb} that
\begin{equation}
\mathcal B(\omega) \geq \mathcal B(\omega) \gamma^{\frac{1}{3}} +\mathcal B(\omega) (1-\gamma)^{\frac{1}{3}},
\end{equation}
where $\gamma := \| U_* \|^6_{L^6(\mathcal{T})}$ and we used the fact that $\mathcal{B}(\omega)$
is the infimum of $B_{\omega}(U)$ with the constraint $\| U \|_{L^6(\mathcal{T})}^6 = 1$.
Hence, $\gamma \in [0,1]$ satisfies the bound
\begin{equation*}
\gamma^{\frac{1}{3}} + (1-\gamma)^{\frac{1}{3}} \leq 1,
\end{equation*}
where the map $x \mapsto f(x) := x^{\frac{1}{3}} + (1-x)^{\frac{1}{3}}$ is such that $f(0)=f(1)=1$ and strictly concave.
Hence either $\gamma = 0$ and $\gamma = 1$.

If $\gamma = 1$, then $B_{\omega}(U_*) \geq \mathcal B(\omega)$. However,
it follows from (\ref{quadratic}) that $\mathcal{B}(\omega) \geq B_{\omega}(U_*)$,
hence $B_{\omega}(U_*) = \mathcal{B}(\omega)$
so that $U_*$ is a minimizer of the constrained problem (\ref{infB}).
\end{proof}

\begin{lemma}
\label{lemma-1-2}
For every $\omega < 0$, it follows that $\gamma = 1$.
\end{lemma}

\begin{proof}
By Lemma \ref{lemma-1-0}, either $\gamma = 0$ or $\gamma = 1$, so we only need to exclude the case $\gamma = 0$.
Let us define the variational problem analogous to \eqref{infB} but posed on the line:
\beq\label{ineqsol2}
\mathcal{B}_{\mathbb{R}}(\omega) := \inf_{w\in H^1(\mathbb{R})} \left\{ B_{\omega}(w;\mathbb{R}) : \quad
\| w \|_{L^{6}(\mathbb{R})} = 1 \right\},
\eeq
where
$$
B_{\omega}(w;\mathbb{R}) :=  \| w' \|^2_{L^2(\mathbb{R})} - \omega \| w \|^2_{L^2(\mathbb{R})}.
$$
As is well known, the infimum $\mathcal{B}_{\mathbb{R}}(\omega)$ is attained at the scaled soliton $\lambda_{\mathbb{R}} \varphi_{\omega}$
satisfying the constraint $\| \lambda_{\mathbb{R}} \varphi_{\omega} \|_{L^6(\mathbb{R})}^6 = 1$, from which it follows that
$$
\lambda_{\mathbb{R}} = \left( \frac{4}{\pi |\omega|} \right)^{1/6}.
$$
Evaluating the integrals in \eqref{ineqsol2} yields the exact expression:
\beq\label{solB}
\mathcal{B}_{\mathbb{R}}(\omega) = \| \lambda_{\mathbb{R}} \varphi_{\omega}' \|_{L^2(\mathbb{R})}^2
- \omega \| \lambda_{\mathbb{R}} \varphi_{\omega} \|_{L^2(\mathbb{R})}^2
= \frac{3 \pi |\omega|}{4} \lambda_{\mathbb{R}}^2 = \frac{3}{4} \left( 2 \pi |\omega| \right)^{2/3}.
\eeq
We first show that if $\gamma = 0$, then the minimizing sequence $\{ U_n \}_{n \in \mathbb{N}}$ escapes
to infinity along the half-line in $\mathcal{T}$ as $n \to \infty$ so that $U_* = 0$ and
$\mathcal{B}(\omega) \geq \mathcal{B}_{\mathbb{R}}(\omega)$. Then, we show that for every $\omega < 0$
there exists a trial function $U_0 \in H^1_{\rm C}(\mathcal{T})$ such that $\| U_0 \|_{L^6(\mathcal{T})} = 1$ and
$B_{\omega}(U_0) < \mathcal{B}_{\mathbb{R}}(\omega)$. Therefore, the minimizing sequence cannot escape to infinity
so that $\gamma \neq 0$. Hence $\gamma = 1$ by Lemma \ref{lemma-1-0}.

To proceed with the first step, let $\{U_n\}_{n\in\NA}$ be a minimizing sequence such that $U_n \in H^1_{\rm C}(\mathcal{T})$,
$\|U_n \|_{L^6(\mathcal{T})} = 1$, $\lim\limits_{n\to\infty} B_{\omega}(U_n) = \mathcal{B}(\omega)$ and 
suppose that $U_n \to 0$ weakly in $H^1(\mathcal T)$.

For simplicity, we consider the nonnegative sequence with $U_n\geq 0$. Let $\epsilon_n \geq 0$ be the maximum of $U_n$ on
$[-\pi,\pi] \cup [0,2\pi] \subset \mathcal{T}$. Since $U_n \to 0$ as $n \to \infty$ uniformly
on any compact subset of $\mathcal{T}$, we have $\epsilon_n \to 0$ as $n \to \infty$. Let us define
$\tilde{U}_n = (\tilde{u}_n,\tilde{v}_n)\in H^1_{\rm C}(\mathcal{T})$ from the components of $U_n = (u_n,v_n)$ as follows:
$$
\tilde{v}_n(x) = \left\{ \begin{array}{ll} v_n(x), \quad & x \in [2 \pi,\infty), \\
v_n(2\pi) \frac{x - \pi}{\pi}, \quad & x \in [\pi,2\pi], \\
2^{1/6} \epsilon_n \frac{\pi - x}{\pi}, \quad & x \in [0,\pi], \end{array} \right.
$$
and
$$
\tilde{u}_n(x) = 2^{1/6} \epsilon_n, \quad x \in [-\pi,\pi].
$$
Since $\| U_n - \tilde{U}_n \|_{H^1(\mathcal{T})} \to 0$ as $n \to \infty$, we have
$\lim\limits_{n\to\infty} B_{\omega}(\tilde{U}_n) = \mathcal{B}(\omega)$. In addition,
$\| \tilde{U}_n \|_{L^6(\mathcal{T})} \geq 1$ because
$$
\| \tilde{U}_n \|^6_{L^6([-\pi,\pi] \cup [0,2\pi])} \geq 4 \pi \epsilon^6_n \geq \| U_n \|^6_{L^6([-\pi,\pi] \cup [0,2\pi])}.
$$
By the proof of Proposition \ref{prop-Appendix-B} in Appendix B, minimizing $B_{\omega}(U)$ under the constraint
$\| U \|_{L^6(\mathcal{T})} = 1$ is the same as minimizing $B_{\omega}(U)$ in $\| U \|_{L^6(\mathcal{T})} \geq 1$,
therefore, $\{ \tilde{U}_n \}_{n\in\NA}$ is also a minimizing sequence for the same variational problem (\ref{infB}).
At the same time, the image of $\tilde{U}_n$ covers all values in $(0,\max_{x \in \mathcal{T}} U_n(x))$ at least twice.
If ${\tilde U_n}^{s}$ is the symmetric rearrangement of ${\tilde U_n}$ on the line $\mathbb{R}$, then
it follows from the Polya--Szeg\"{o} inequality on graphs (see Proposition 3.1 in \cite{AST15}), that
$$
B_{\omega}(\tilde{U}_n) \geq B_{\omega}(\tilde{U}^{s}_n;\mathbb{R}) \geq \mathcal{B}_{\mathbb{R}}(\omega), \quad n \in \NA.
$$
By taking the limit $n \to \infty$, we obtain $\mathcal{B}(\omega) \geq \mathcal{B}_{\mathbb{R}}(\omega)$
in the case of $\gamma = 0$.

To proceed with the second step, we construct a trial function $U_0 \in H^1_{\rm C}(\mathcal{T})$ such that $\| U_0 \|_{L^6(\mathcal{T})} = 1$ and
$B_{\omega}(U_0) < \mathcal{B}_{\mathbb{R}}(\omega)$ with the following explicit computation. For every $\omega < 0$,
we define
$$
U_0 = \left\{ \begin{array}{ll} \lambda_0 \varphi_{\omega}(x), & \quad x \in [-\pi,\pi], \\
\lambda_0 \varphi_{\omega}(x+\pi), & \quad x \in (0,\infty), \end{array} \right.
$$
hence $U_0$ on $\mathcal{T}$ is a scaled soliton $\lambda_0 \varphi_{\omega}$ truncated on $[-\pi,\infty)$.
Then, $\lambda_0$ is found from the normalization condition:
\begin{eqnarray*}
1 & = & \frac{1}{2} \lambda_0^6 |\omega| \int_{-2 \pi |\omega|^{1/2}}^{\infty} {\rm sech}^3z \;dz \\
& = & \frac{1}{2} \lambda_0^6 |\omega| \left[ \frac{\pi}{4} + \frac{\sinh(2 \pi |\omega|^{1/2})}{2 \cosh^2(2 \pi |\omega|^{1/2})}
+ \frac{1}{2} \arctan \sinh(2 \pi |\omega|^{1/2}) \right],
\end{eqnarray*}
while we compute that
\begin{eqnarray*}
B_{\omega}(U_0) & = & \frac{1}{2} \lambda_0^2 |\omega| \int_{-2 \pi |\omega|^{1/2}}^{\infty} \left[ 2 {\rm sech}z - {\rm sech}^3z \right] dz \\
& = & \frac{1}{2} \lambda_0^2 |\omega| \left[ \frac{3 \pi}{4} - \frac{\sinh(2 \pi |\omega|^{1/2})}{2 \cosh^2(2 \pi |\omega|^{1/2})}
+ \frac{3}{2} \arctan \sinh(2 \pi |\omega|^{1/2}) \right] \\
& = & \frac{3}{4} (\pi |\omega|)^{2/3} f(2 \pi |\omega|^{1/2}),
\end{eqnarray*}
where
\begin{eqnarray*}
f(A) := \frac{1 + \frac{2}{\pi} \arctan \sinh(A) - \frac{2 \sinh(A)}{3 \pi \cosh^2(A)}}{
\left[ 1 + \frac{2}{\pi} \arctan \sinh(A) + \frac{2 \sinh(A)}{\pi \cosh^2(A)} \right]^{1/3}}, \quad A := 2 \pi |\omega|^{1/2}.
\end{eqnarray*}
It is clear that $f(0) = 1$ and $\lim_{A \to \infty} f(A) = 2^{2/3}$. We shall prove that $f(A) < 2^{2/3}$ for every $A > 0$.
Indeed, for every $A > 0$,
$$
f(A) \leq \left[ 1 + \frac{2}{\pi} \arctan \sinh(A) + \frac{2 \sinh(A)}{\pi \cosh^2(A)} \right]^{2/3} =: \left[ g(\sinh(A)) \right]^{2/3},
$$
where
$$
g(z) := 1 + \frac{2}{\pi} \arctan z + \frac{2 z}{\pi (1+z^2)}, \quad z := \sinh(A).
$$
Since
$$
g'(z) = \frac{4}{\pi (1+z^2)^2} > 0,
$$
$g$ is monotonically increasing on $\mathbb{R}^+$ so that
$$
f(A) \leq [g(\sinh(A))]^{2/3} < \left[ \lim_{A \to \infty} g(\sinh(A)) \right]^{2/3} = 2^{2/3}.
$$
Thus, $B_{\omega}(U_0) < \mathcal{B}_{\mathbb{R}}(\omega)$ for every $\omega < 0$.

Both steps are complete and $\gamma = 0$ is impossible for the minimizing sequence $\{ U_n \}_{n \in \mathbb{N}}$.
\end{proof}

\begin{remark}
Any smooth and compactly supported function in $H^1(\RE)$ can be considered as an element of
$H^1_{\rm C}(\mathcal{T})$ (see the proof of Theorem 2.2 of \cite{AST15} and Remark 2.2 in \cite{AST17}),
so that by a density argument we have (independently on $\gamma$)
 \beq\label{ineqsol1}
\mathcal{B}(\omega) \leq \inf_{w\in H^1(\mathbb{R})} \left\{ B_{\omega}(w;\mathbb{R}) : \quad
\| w \|_{L^{6}(\mathbb{R})} = 1 \right\}
= \mathcal{B}_{\mathbb{R}}(\omega).
\eeq
Hence, for the escaping minimizing sequence with $\gamma = 0$ we would actually have that $\mathcal{B}(\omega) = \mathcal{B}_{\mathbb{R}}(\omega)$.
However, the existence of the trial function $U_0 \in H^1_{\rm C}(\mathcal{T})$ such that $\| U_0 \|_{L^6(\mathcal{T})} = 1$ and
$B_{\omega}(U_0) < \mathcal{B}_{\mathbb{R}}(\omega)$ eliminates the case $\gamma = 0$.
\end{remark}

It follows from Lemmas \ref{lemma-1-0} and \ref{lemma-1-2} that there exists a global minimizer $\Psi \in H^1_{\rm C}(\mathcal{T})$
of the variational problem (\ref{infB}) for every $\omega < 0$.
We now verify properties of the global minimizer $\Psi \in H^1_{\rm C}(\mathcal{T})$.

\begin{lemma}
\label{lemma-1-3}
Let $\Psi \in H^1_{\rm C}(\mathcal{T})$ be the global minimizer of the variational problem (\ref{infB}) for every $\omega < 0$.
Then, $\Psi$ is real up to the phase rotation, positive up to the sign choice, symmetric on $[-\pi,\pi]$, and monotonically decreasing
on $[0,\pi]$ and $[0,\infty)$.
\end{lemma}

\begin{proof}
If $\Psi$ is a minimizer of the variational problem (\ref{infB}), so is $|\Psi|$. Hence we may assume that $\Psi$ is real and positive.
To prove symmetry and monotonic decay, we observe that if the minimizer is not symmetric on $[-\pi,\pi]$
and is not decreasing on $[0,\pi]$ and $[0,\infty)$, then it is possible to define a suitable competitor on the tadpole
graph $\mathcal{T}$ with lower value of $\mathcal{B}(\omega)$, by using the well known technique of the symmetric rearrangements
and the Polya--Szeg\"{o} inequality on graphs (see Proposition 3.1 in \cite{AST15}, examples discussed after
Corollary 3.4 in \cite{AST16} and in \cite{D19}).
\end{proof}

The following lemma gives as further information with a more precise quantitative control of the infimum $\mathcal{B}(\omega)$.
This result is similar to the energy bounds in (\ref{bounds-on-E}) obtained for the subcritical nonlinearity.

\begin{lemma}
\label{lemma-1-4}
For every $\omega < 0$, the infimum $\mathcal{B}(\omega)$ in (\ref{infB}) satisfies the bounds
\begin{equation}
\label{bounds-on-B}
\mathcal{B}_{\mathbb{R}^+}(\omega) < \mathcal{B}(\omega) < \mathcal{B}_{\mathbb{R}}(\omega),
\end{equation}
where
\beq\label{inferreerreplus}
\mathcal{B}_{\mathbb{R}^+}(\omega) = \frac{3}{4} \left( \pi |\omega| \right)^{2/3}, \quad
\mathcal{B}_{\mathbb{R}}(\omega) = \frac{3}{4} \left( 2 \pi |\omega| \right)^{2/3}.
\eeq
\end{lemma}

\begin{proof}
The upper bound in (\ref{bounds-on-B}) is verified in the proof of Lemma \ref{lemma-1-2}. In order to prove
the lower bound, we use the Gagliardo--Nirenberg inequality on graphs:
$$
\| U \|^6_{L^6(\mathcal{T})} \leq K_{\mathcal{T}} \| U \|^4_{L^2(\mathcal{T})} \| U' \|^2_{L^2(\mathcal{T})},
$$
with
$$
K_{\mathcal{T}} := \sup_{U \in H^1_{\rm C}(\mathcal{T}) : \;\; U \neq 0} \frac{\| U \|^6_{L^6(\mathcal{T})}}{
\| U \|^4_{L^2(\mathcal{T})} \| U' \|^2_{L^2(\mathcal{T})}} =
\left(\inf_{U \in H^1_{\rm C}(\mathcal{T}) : \;\; \| U \|_{L^6(\mathcal{T})} = 1}
\| U \|^4_{L^2(\mathcal{T})} \| U' \|^2_{L^2(\mathcal{T})} \right)^{-1}.
$$
By Theorem 3.3 in \cite{AST17}, it follows that $K_{\mathcal{T}} = K_{\mathbb{R}^+}$
and the constant $K_{\mathbb{R}^+}$ is attained by the half soliton
$\lambda_{\mathbb{R}^+} \varphi_{\omega}$ normalized by $\| \lambda_{\mathbb{R}^+} \varphi_{\omega} \|_{L^6(\mathbb{R}^+)} = 1$.
By similar computations as in the proof of Lemma \ref{lemma-1-2}, we obtain that
$$
\lambda_{\mathbb{R}^+} = \left( \frac{8}{\pi |\omega|} \right)^{1/6} \quad \mbox{\rm and} \quad
K_{\mathbb{R}^+} = \frac{16}{\pi^2},
$$
from which it also follows that
$$
\mathcal{B}_{\mathbb{R}^+}(\omega) := \| \lambda_{\mathbb{R}^+} \varphi_{\omega}' \|_{L^2(\mathbb{R}^+)}^2
- \omega \| \lambda_{\mathbb{R}^+} \varphi_{\omega} \|_{L^2(\mathbb{R}^+)}^2 =
\frac{3}{4} \left( \pi |\omega| \right)^{2/3}.
$$
By using the definition and the value of $K_{\mathcal{T}}$, we obtain for every $U \in H^1_C(\mathcal{T})$,
$$
\| U' \|^2_{L^2(\mathcal{T})} \geq \frac{\pi^2}{16 \| U \|^4_{L^2(\mathcal{T})}}
$$
so that
$$
B_{\omega}(U) \geq \frac{\pi^2}{16 \| U \|^4_{L^2(\mathcal{T})}} + |\omega| \| U \|^2_{L^2(\mathcal{T})} \geq
\frac{3}{4} \left( \pi |\omega| \right)^{2/3}.
$$
where the latter inequality follows from the minimization
of $f(x) := \frac{\pi^2}{16 x^2} + |\omega| x$ in $x$ on $\mathbb{R}^+$.
Hence $B_{\omega}(U) \geq \mathcal{B}_{\mathbb{R}^+}(\omega)$ for every $U \in H^1_C(\mathcal{T})$. Since $\mathcal{T}$
is not isometric to $\mathcal{R}^+$, it follows from the proof
of Theorem 3.3 in \cite{AST17} that the equality cannot be attained
on $\mathcal{T}$, hence $\mathcal{B}(\omega) > \mathcal{B}_{\mathbb{R}^+}(\omega)$.
\end{proof}

Assuming that $\Psi \in H^1_{\rm C}(\mathcal{T})$ is a global minimizer of
the variational problem (\ref{infB}), we show that it yields a solution
$\Phi \in H^2_{\rm NK}(\mathcal{T})$ to the stationary NLS equation (\ref{eqstaz}).

\begin{lemma}
\label{lemma-1-1}
Let $\Psi \in H^1_{\rm C}(\mathcal{T})$ be a minimizer of
the variational problem (\ref{infB}). Then $\Psi \in H^2_{\rm NK}(\mathcal{T})$
and $\Phi := \left( \frac{1}{3} \mathcal{B}_{\omega} \right)^{1/4} \Psi$
is a solution to the stationary NLS equation (\ref{eqstaz}).
\end{lemma}

\begin{proof}
By using Lagrange multipliers in $\Sigma_{\omega,\nu}(U) := B_{\omega}(U) - \nu \| U \|_{L^{6}(\mathcal{T})}^{6}$,
we obtain Euler--Lagrange equation for $\Psi$:
\beq \label{eq-Psi}
-\Delta \Psi - 3 \nu \Psi^5 = \omega \Psi.
\eeq
Since $H^1_{\rm C}(\mathcal{T})$ is a Banach algebra with respect to multiplication,
$\Psi^5 \in H^1_{\rm C}(\mathcal{T})$, so that one can rewrite the
Euler--Lagrange equation (\ref{eq-Psi}) in the form
$\Psi = 3 \nu (-\Delta - \omega)^{-1} \Psi^5$, where
$(-\Delta - \omega)^{-1} : L^2(\mathcal{T}) \mapsto H^2(\mathcal{T})$ is a bounded operator
thanks to $\omega < 0$ and $\sigma(-\Delta) \geq 0$. Hence, the same solution
$\Psi$ is actually in $H^2(\mathcal{T})$. Since the boundary conditions
(\ref{dd}) are {\em natural} boundary conditions for integration by parts,
it then follows that $\Psi \in H^2(\mathcal{T}) \cap H^1_{\rm C}(\mathcal{T})$
satisfies the boundary conditions (\ref{dd}) so that $\Psi \in H^2_{\rm NK}(\mathcal{T})$.

It follows from the constraint in (\ref{Bfunctional}) that
$\nu = \frac{1}{3} \mathcal{B}_{\omega} > 0$ since $\| \Psi \|^6_{L^6(\mathcal{T})} = 1$.
The scaled function $\Phi = \nu^{1/4} \Psi$ satisfies the stationary NLS equation (\ref{eqstaz}).
\end{proof}

Lemmas  \ref{lemma-1-0}, \ref{lemma-1-2}, \ref{lemma-1-3}, and \ref{lemma-1-1}
yield the proof of Theorem \ref{theorem-existence}.

\section{Dynamical system methods for the standing waves}
\label{sec-dynamical}

Here we shall prove Theorem \ref{theorem-degeneracy}. We do so by using the dynamical system
methods for characterization of the standing wave $\Phi \in H^2_{\rm NK}(\mathcal{T})$
of Theorem \ref{theorem-existence}. In particular, we reduce the stationary NLS
equation (\ref{eqstaz}) to the second-order differential equation on an interval,
for which we introduce {\em the period function}. By using the analytical theory for differential
equations, we show monotonicity
of the period function, which allows us to control nullity of the
linearization operator $\mathcal{L}$ in (\ref{Jacobian}).

Let $\Phi \in H^2_{\rm NK}(\mathcal{T})$ be a real and positive solution to the stationary
NLS equation (\ref{eqstaz}) with $\omega < 0$ constructed by Theorem \ref{theorem-existence}.
For every $\omega < 0$, we set $\omega = -\varepsilon^4$ and introduce the scaling
transformation for $\Phi = (u,v)$ as follows:
\begin{equation}
\left\{ \begin{array}{ll}
u(x) = \varepsilon U(\varepsilon^2 x), \quad & x \in [-\pi,\pi], \\
v(x) = \varepsilon V(\varepsilon^2 x), \quad & x \in [0,\infty). \end{array} \right.
\label{scaling}
\end{equation}
The boundary-value problem for $(U,V)$ is rewritten in the component form:
\begin{equation}\label{sys}
\left\{\begin{array}{ll}
- U'' + U - 3 U^5 = 0, \quad & z \in (-\pi \varepsilon^2, \pi \varepsilon^2), \\
- V'' + V - 3 V^5 = 0, \quad & z \in (0, \infty),\\
U(\pi \varepsilon^2) = U(-\pi \varepsilon^2) = V(0), & \\
U'(\pi \varepsilon^2) - U'(-\pi \varepsilon^2) = V'(0).&
\end{array}\right.
\end{equation}
By the symmetry property in Theorem \ref{theorem-existence}, we have $U(-z) = U(z)$, $z \in [-\pi \varepsilon^2, \pi \varepsilon^2]$.
By uniqueness of the soliton $\varphi$ on the half-line up to the spatial translation,
we have $V(z) = \varphi(z+a)$, $z \in [0,\infty)$ for some $a \in \mathbb{R}$,
where $\varphi(z) = {\rm sech}^{1/2}(2 z)$. By the monotonicity property in Theorem \ref{theorem-existence},
we have $a \in (0,\infty)$. These simplifications allow us to reduce
the existence problem (\ref{sys}) to the simplified form
\begin{equation}\label{sys1}
\left\{\begin{array}{l}
- U'' + U - 3 U^5 = 0, \quad z \in (0, \pi \varepsilon^2), \\
U'(0) = 0, \\
U(\pi \varepsilon^2) = \varphi(a), \\
2 U'(\pi \varepsilon^2) = \varphi'(a),
\end{array}\right.
\end{equation}
where $a \in (0,\infty)$ and $\varepsilon \in (0,\infty)$.

Many stationary states can be represented by solutions of the boundary-value problem (\ref{sys1}).
However, the monotonicity property in Theorem \ref{theorem-existence} allows us to reduce our consideration
to the unique monotonically decreasing solution $U$ on $[0,\pi \varepsilon^2]$ shown by
blue line on the phase plane $(U,U')$ (Figure \ref{fig-phase}). By the boundary conditions
in the system (\ref{sys1}), the solution is related to the monotonically decreasing part of the homoclinic orbit
shown by red line on the phase plane $(U,U')$ in such a way that
the value of $U$ at the vertex is continuous, where the value of $U'$ at the vertex jumps by half of its value.
The green line is an image of the homoclinic orbit after $U'$ is reduced by half.
The value of $U$ at the vertex is adjusted depending on the value of $\varepsilon$ in the length of the interval $[0,\pi \varepsilon^2]$.

\begin{figure}[h!]
	\centering
	\includegraphics[width=12cm,height=8cm]{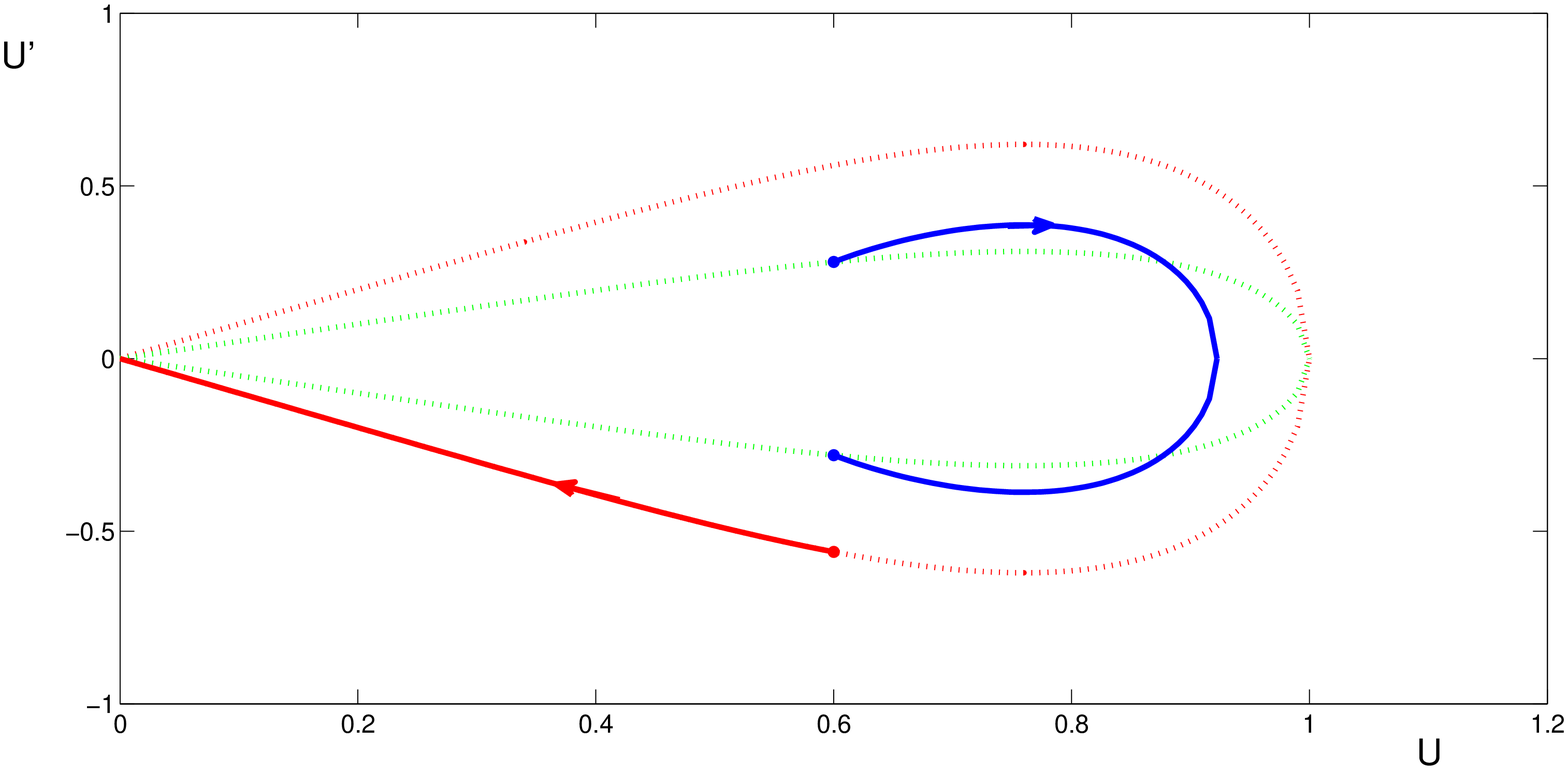}
	\caption{Representation of the solution to the boundary-value problem (\ref{sys}) on the phase plane.}
	\label{fig-phase}
\end{figure}

The phase plane representation of the solution to the boundary-value problem (\ref{sys})
shown on Fig. \ref{fig-phase} is made rigorous in the following lemma.

\begin{lemma}
\label{lemma-2-1}
For every $a > 0$ there exists a unique value of $\varepsilon > 0$ for which
there exists a unique solution $U \in C^2(0,\pi \varepsilon^2)$
to the boundary-value problem (\ref{sys1}) such that $U$ is monotonically decreasing on $[0,\pi\varepsilon^2]$.
Moreover, the map $(0,\infty) \ni a \mapsto \varepsilon(a) \in (0,\infty)$ is $C^1$ and monotonically increasing.
\end{lemma}

\begin{proof}
The differential equation $-U'' + U - 3 U^5 = 0$ can be solved in quadrature with the first-order invariant:
\begin{equation}
\label{first}
E = (U')^2 - U^2 + U^6.
\end{equation}
Since the value of $E$ is constant in $z$, we obtain the exact value of $E$ for the admissible solution
to the system (\ref{sys1}):
\begin{equation}
\label{a-to-energy}
E = \frac{1}{4} [\varphi'(a)]^2 - \varphi(a)^2 + \varphi(a)^6 = -\frac{3}{4} [\varphi'(a)]^2 < 0.
\end{equation}
Let $U_0(a) := \varphi(a)$, so that the map $(0,\infty) \ni a \mapsto U_0(a) \in (0,1)$ is $C^1$ and monotonically decreasing.
Let $U_+(a)$ be the largest positive root of $E + U^2 - U^6 = 0$ such that $U_+(a) \geq U_* := \frac{1}{3^{1/4}}$,
which exists if $E \in (E_0,0)$, where $E_0 := -\frac{2}{3 \sqrt{3}}$. It follows from
(\ref{a-to-energy}) that the latter requirement is satisfied for every $a \in (0,\infty)$.
By means of the first-order invariant (\ref{first}), the boundary-value problem
(\ref{sys1}) is solved in the following quadrature:
\begin{equation}
\label{quadrature}
\pi \varepsilon^2 = \int_{U_0}^{U_+} \frac{du}{\sqrt{E + u^2 - u^6}}.
\end{equation}
Since $E$, $U_0$, and $U_+$ are uniquely defined by $a \in (0,\infty)$,
the value of $\varepsilon$ is uniquely defined by (\ref{quadrature}) from the value of $a \in (0,\infty)$.

It remains to prove that the map $(0,1) \ni U_0 \mapsto \varepsilon(U_0) \in (0,\infty)$ is $C^1$ and monotonically decreasing.
Since the map $(0,\infty) \ni a \mapsto U_0(a) \in (0,1)$ is $C^1$ and monotonically decreasing,
the two results imply that the composite map the map $(0,\infty) \ni a \mapsto \varepsilon(a) \in (0,\infty)$ is $C^1$ and monotonically increasing,
which yields the assertion of the lemma with the solution $U$ given in the implicit form by the quadrature (\ref{quadrature}).

To prove monotonicity of the $C^1$ map $(0,1) \ni U_0 \mapsto \varepsilon(U_0) \in (0,\infty)$,
we use the technique developed for the flower graph in the cubic NLS equation \cite{KMPZ}. We define
the following period function:
\begin{equation}
\label{period-function}
T(U_0) := \int_{U_0}^{U_+} \frac{du}{\sqrt{E + A(u)}}, \quad U_0 \in (0,1),
\end{equation}
where $A(u) = u^2 - u^6$, $U_+$ is the largest positive root of $E + A(u) = 0$ such that $U_+ \geq U_* := \frac{1}{3^{1/4}}$,
and $E$ is given by $E = \frac{1}{4} A(U_0) - A(U_0) = -\frac{3}{4} A(U_0)$. The value $U_*$ is the only critical (maximum) point of $A(U)$
on $\mathbb{R}^+$ with $A(U_*) = \frac{2}{3 \sqrt{3}}$ and $A'(U_*) = 0$.

Recall that if $W(u,v)$ is a $C^1$ function in an open region of $\mathbb{R}^2$, then the differential of $W$ is defined by 
$$
dW(u,v) = \frac{\partial W}{\partial u} du + \frac{\partial W}{\partial v} dv
$$
and the line integral of $d W(u,v)$ along any $C^1$ contour $\gamma$ connecting two points $(u_0,v_0)$ and $(u_1,v_1)$ does not depend on $\gamma$ and is evaluated as 
$$
\int_{\gamma} d W(u,v) = W(u_1,v_1) - W(u_0,v_0).
$$

Define $p := \sqrt{E + A(u)}$ and compute for every $u \in (0,1)$:
\begin{eqnarray*}
d \left( \frac{2 p [A(u) - A(U_*)]}{A'(u)} \right) & = &
2 \left[ 1 - \frac{A''(u) [A(u) - A(U_*)]}{[A'(u)]^2} \right] p du + \frac{2[ A(u) - A(U_*)]}{A'(u)} dp \\
& = & 2 \left[ 1 - \frac{A''(u) [A(u) - A(U_*)]}{[A'(u)]^2} \right] p du + \frac{A(u) - A(U_*)}{p} du.
\end{eqnarray*}
All terms in this expression are non-singular for every $u \in (0,1)$. It enables us to express the period function $T(U_0)$ in the equivalent way:
\begin{eqnarray*}
[E + A(U_*)] T(U_0) & = & \int_{U_0}^{U_+} p du - \int_{U_0}^{U_+} \frac{A(u) - A(U_*)}{p} du \\
& = & \int_{U_0}^{U_+} \left[ 3 - 2 \frac{A''(u) [A(u) - A(U_*)]}{[A'(u)]^2} \right] p du + \frac{A(U_0) - A(U_*)}{A'(U_0)} \sqrt{A(U_0)},
\end{eqnarray*}
where we have used that $2 \sqrt{E + A(U_0)} = \sqrt{A(U_0)}$. The right-hand side is $C^1$ in $U_0$ on $(0,1)$,
which proves that the map $(0,1) \ni U_0 \mapsto \varepsilon(U_0) \in (0,\infty)$ is $C^1$ in $U_0$.
Moreover, we compute the derivative explicitly by
\begin{eqnarray}
\nonumber
[E + A(U_*)] T'(U_0) & = & -\frac{3}{8} A'(U_0) \int_{U_0}^{U_+} \left[ 1 - 2 \frac{A''(u) [A(u) - A(U_*)]}{[A'(u)]^2} \right] \frac{du}{p} \\
& \phantom{t} & - \frac{A(U_*)}{2\sqrt{A(U_0)}},
\label{derivative-T}
\end{eqnarray}
where we have used that $\frac{dE}{dU_0} = -\frac{3}{4} A'(U_0)$.

We need to prove that $T'(U_0) < 0$. Note that $E + A(U_*) > 0$ and the last term in the right-hand side of (\ref{derivative-T}) is negative.
In order to analyze the first term in the right-hand side of (\ref{derivative-T}), we compute directly
\begin{eqnarray}
\nonumber
1 - 2 \frac{A''(u) [A(u) - A(U_*)]}{[A'(u)]^2} & = &
1 + \frac{(1-15 u^4) (2 - 3 \sqrt{3} u^2 + 3 \sqrt{3} u^6)}{3 \sqrt{3} u^2 (1-\sqrt{3} u^2)^2 (1 + \sqrt{3} u^2)^2} \\
\nonumber
& = & 1 + \frac{(1-15 u^4) (2 + \sqrt{3} u^2)}{3 \sqrt{3} u^2 (1 + \sqrt{3} u^2)^2} \\
& = & \frac{2 (1 - \sqrt{3} u^2) (1 + 3 \sqrt{3} u^2 + 3 u^4)}{3 \sqrt{3} u^2 (1 + \sqrt{3} u^2)^2}.
\label{derivative-T-term}
\end{eqnarray}
If $U_0 \in (U_*,1)$, then $A'(U_0) < 0$ and the right-hand side of (\ref{derivative-T-term}) is negative for every $u \in [U_0,U_+]$.
Hence the first term in the right-hand side of (\ref{derivative-T}) is negative and so is $T'(U_0)$ if $U_0 \in (U_*,1)$.
Thus, $T'(U_0) < 0$ for $U_0 \in (U_*,1)$. Since $A'(U_*) = 0$, we also have $T'(U_*) < 0$.

In order to study $T'(U_0)$ for $U_0 \in (0,U_*)$, for which $A'(U_0) > 0$,
we need to integrate the expression in (\ref{derivative-T-term}) by parts:
\begin{eqnarray*}
& \phantom{t} & \int_{U_0}^{U_+} \frac{2 (1 - \sqrt{3} u^2) (1 + 3 \sqrt{3} u^2 + 3 u^4)}{3 \sqrt{3} u^2 (1 + \sqrt{3} u^2)^2} \frac{du}{\sqrt{E + A(u)}} =
\int_{U_0}^{U_+} \frac{1 + 3 \sqrt{3} u^2 + 3 u^4}{3 \sqrt{3} u^3 (1 + \sqrt{3} u^2)^3} \frac{ A'(u) du}{\sqrt{E + A(u)}} \\
& \phantom{t} &  = - \frac{1 + 3 \sqrt{3} U_0^2 + 3 U_0^4}{3 \sqrt{3} U_0^3 (1 + \sqrt{3} U_0^2)^3} \sqrt{A(U_0)}
+ 2 \int_{U_0}^{U_+} \frac{1 + 4 \sqrt{3} u^2 + 20 u^4 + 5 \sqrt{3} u^6}{\sqrt{3} u^4 (1 + \sqrt{3} u^2)^4} \sqrt{E + A(u)} du,
\end{eqnarray*}
where the first term is negative and the last term is positive. Recall that $A'(U_0) > 0$ for $U_0 \in (0,U_*)$.
Combining the last negative term in the right-hand side of (\ref{derivative-T})
and the first positive term obtained from the previous expression yields
\begin{eqnarray*}
& \phantom{t} & \frac{1 + 3 \sqrt{3} U_0^2 + 3 U_0^4}{8 \sqrt{3} U_0^3 (1 + \sqrt{3} U_0^2)^3} A'(U_0) \sqrt{A(U_0)}
- \frac{A(U_*)}{2\sqrt{A(U_0)}} \\
& \phantom{t} & = -\frac{1 + 2 \sqrt{3} U_0^2 + 33 U_0^4 + 15 \sqrt{3} U_0^6 - 18 U_0^8 - 9 \sqrt{3} U_0^{10}}{
12 \sqrt{3} \sqrt{A(U_0)} (1 + \sqrt{3} U_0^2)^2} \\
& \phantom{t} & = -\frac{1 + 2 \sqrt{3} U_0^2 + 15 U_0^4 +
18 U_0^4 (1 - U_0^4) + 6 \sqrt{3} U_0^6 + 9 \sqrt{3} U_0^6 (1 - U_0^4)}{12 \sqrt{3} \sqrt{A(U_0)} (1 + \sqrt{3} U_0^2)^2},
\end{eqnarray*}
which is negative for every $U_0 \in (0,U_*)$. Hence the right-hand side of (\ref{derivative-T}) is negative
and so is $T'(U_0)$ for every $U_0 \in (0,U_*)$. Thus, we have proven that $T'(U_0) < 0$ for every $U_0 \in (0,1)$,
from which it follows that the map $(0,1) \ni U_0 \mapsto \varepsilon(U_0) \in (0,\infty)$ is monotonically decreasing.

Finally, we check the asymptotic limits
$$
\lim_{U_0 \to 1} T(U_0) = 0 \quad \mbox{\rm and} \quad \lim_{U_0 \to 0} T(U_0) = \infty,
$$
which imply that  the map $(0,1) \ni U_0 \mapsto \varepsilon(U_0) \in (0,\infty)$ is onto. Indeed, as $U_0 \to 1$,
we have $U_+ \to 1$ so that $|U_+-U_0| \to 0$ as $U_0 \to 1$. Since the weakly singular integral is integrable,
we have
$$
T(U_0) = \int_{U_0}^{U_+} \frac{du}{\sqrt{A(u) - A(U_+)}} \to 0 \quad \mbox{\rm as} \quad U_0 \to 1.
$$
On the other hand, the period function $T(U_0)$ is estimated from below for every $0 < U_0 < U_+ < 1$ by
$$
T(U_0) = \int_{U_0}^{U_+} \frac{du}{\sqrt{A(u) - A(U_+)}} \geq \int_{U_0}^{U_+} \frac{du}{u \sqrt{1 - u^4}},
$$
and since $U_+ \to 1$ as $U_0 \to 0$, the lower bound diverges as $U_0 \to 0$.
\end{proof}

The following lemma characterizes the nullity index $z(\mathcal{L})$ of the linearized operator
$\mathcal{L} : H^2_{\rm NK}(\mathcal{T}) \subset L^2(\mathcal{T}) \mapsto L^2(\mathcal{T})$
given by (\ref{Jacobian})  for every $\omega < 0$.

\begin{lemma}
\label{lemma-3-1}
Let $\Phi \in H^2_{\rm NK}(\mathcal{T})$ be a solution to the stationary
NLS equation (\ref{eqstaz}) with $\omega < 0$ defined by (\ref{scaling})
with fixed $\varepsilon > 0$.
Then, $\sigma_{ac}(\mathcal{L}) = [|\omega|,\infty)$ and $z(\mathcal{L}) = 0$.
\end{lemma}

\begin{proof}
Let us consider the spectral problem $\mathcal{L} \Upsilon = \lambda \Upsilon$.
By Weyl's Theorem, thanks to the exponential decay of $v(x) \to 0$ as $x \to \infty$,
$\sigma_{ac}(\mathcal{L}) = \sigma_{ac}(-\Delta - \omega) = [|\omega|,\infty)$.
Therefore, $\lambda = 0$ is isolated from the absolute continuous spectrum of $\mathcal{L}$.

We consider the most general solution of $\mathcal{L} \Upsilon = 0$ and
prove that $\Upsilon \notin H^2_{NK}(\mathcal{T})$ for every $\omega < 0$.
We use the representation $\omega = -\varepsilon^4$ and the scaling
transformation (\ref{scaling}) for $\Phi = (u,v) \in H^2_{\rm NK}(\mathcal{T})$.
Similarly, we represent $\Upsilon = (\mathfrak{u},\mathfrak{v})$ by using the scaling
transformation
\begin{equation}
\left\{ \begin{array}{ll}
\mathfrak{u}(x) = \mathfrak{U}(\varepsilon^2 x), \quad & x \in [-\pi,\pi], \\
\mathfrak{v}(x) = \mathfrak{V}(\varepsilon^2 x), \quad & x \in [0,\infty), \end{array} \right.
\label{scaling-lin}
\end{equation}
from which the following boundary-value problem is obtained for $(\mathfrak{U},\mathfrak{V})$:
\begin{equation}\label{sys-lin}
\left\{\begin{array}{ll}
- \mathfrak{U}'' + \mathfrak{U} - 15 U^4 \mathfrak{U} = 0, \quad & z \in (-\pi \varepsilon^2, \pi \varepsilon^2), \\
- \mathfrak{V}'' + \mathfrak{V} - 15 V^4 \mathfrak{V} = 0, \quad & z \in (0, \infty),\\
\mathfrak{U}(\pi \varepsilon^2) = \mathfrak{U}(-\pi \varepsilon^2) = \mathfrak{V}(0), &\\
\mathfrak{U}'(\pi \varepsilon^2) - \mathfrak{U}'(-\pi \varepsilon^2) = \mathfrak{V}'(0). &
\end{array}\right.
\end{equation}

We are looking for a solution $(\mathfrak{U},\mathfrak{V}) \in H^2_{NK}(\mathcal{T})$
to the boundary-value problem (\ref{sys-lin}) so that $\mathfrak{V}(z) \to 0$ as $z \to \infty$.
Recall that $V(z) = \varphi(z+a)$ with $a \in (0,\infty)$ defined uniquely
in terms of $\varepsilon \in (0,\infty)$. Then, the only decaying solution
to the second equation in the system (\ref{sys-lin}) takes the form:
\begin{equation}
\label{solution-v}
\mathfrak{V}(z) = \alpha \varphi'(z+a),
\end{equation}
where $\alpha \in \mathbb{C}$ is arbitrary. The general solution
to the first equation in the system (\ref{sys-lin}) can be written in the form:
\begin{equation}
\label{solution-u}
\mathfrak{U}(z) = \beta U'(z) + \gamma W(z),
\end{equation}
where $\beta,\gamma \in \mathbb{C}$ are arbitrary and $W$ is a linearly independent
solution to $U'$. Thanks to the symmetry of the coefficients to the first equation in the system
(\ref{sys-lin}), $W(-z) = W(z)$ and $U'(-z) = -U'(z)$. By using the boundary conditions
in the system (\ref{sys-lin}), we obtain the linear system on coefficients
of the solutions (\ref{solution-v}) and (\ref{solution-u}):
\begin{equation}
\label{system-solution}
\left\{ \begin{array}{l} \beta U'(\pi \varepsilon^2) + \gamma W(\pi \varepsilon^2) = \alpha \varphi'(a), \\
-\beta U'(\pi \varepsilon^2) + \gamma W(\pi \varepsilon^2) = \alpha \varphi'(a),  \\
2 \gamma W'(\pi \varepsilon^2) = \alpha \varphi''(a),
\end{array} \right.
\end{equation}
Since $U'(\pi \varepsilon^2) = \frac{1}{2} \varphi'(a) \neq 0$ for every $a \in (0,\infty)$,
it follows from the system (\ref{system-solution}) that $\beta = 0$ and
a nonzero solution for $(\alpha,\gamma)$ exists if and only if
\begin{equation}
\label{BC-balance}
W(\pi \varepsilon^2) \neq 0 \quad \mbox{\rm and} \quad
\frac{2 W'(\pi \varepsilon^2)}{W(\pi \varepsilon^2)} = \frac{\varphi''(a)}{\varphi'(a)}.
\end{equation}

We shall now express the even solution $W$ to $- \mathfrak{U}'' + \mathfrak{U} - 15 U^4 \mathfrak{U} = 0$.
Let $U(z;E)$ be an even solution of the first-order invariant (\ref{first})
with free parameter $E < 0$ normalized by the boundary condition $U(0;E) = U_+(E)$,
where $U_+(E)$ is the largest positive root of $E + U^2 - U^6 = 0$ such that $U_+(E) \geq U_* := \frac{1}{3^{1/4}}$.
Let $E(\varepsilon)$ be defined for every $\varepsilon > 0$ by the boundary conditions:
\begin{equation}
\label{BC-U}
\left\{
\begin{array}{l}
U(\pi \varepsilon^2; E(\varepsilon)) = \varphi(a), \\
U'(\pi \varepsilon^2; E(\varepsilon)) = \frac{1}{2} \varphi'(a),
\end{array} \right.
\end{equation}
which means that $E(\varepsilon) = -\frac{3}{4} \left[ \varphi'(a)\right]^2$
in accordance with (\ref{a-to-energy}), where
$a \in (0,\infty)$ is uniquely defined from $\varepsilon \in (0,\infty)$ by Lemma \ref{lemma-2-1}.

Since $U(z;E)$ satisfies the second-order equation $-U'' + U - 3 U^5 = 0$,
it follows that $W(z) := \partial_E U(z;E(\varepsilon))$ satisfies the equation
$- \mathfrak{U}'' + \mathfrak{U} - 15 U^4 \mathfrak{U} = 0$ with $U \equiv U(z;E(\varepsilon))$.
Since $E$ is a $C^1$ function of $a$ in (\ref{a-to-energy})
and $a$ is a $C^1$ function of $\varepsilon$ obtained by inverting the monotone $C^1$
mapping $a \mapsto \varepsilon(a)$ in Lemma \ref{lemma-2-1},
we have that $E$ is a $C^1$ function of $\varepsilon$. Differentiating
the boundary conditions (\ref{BC-U}) in $\varepsilon$, we obtain
\begin{equation}
\label{BC-W}
\left\{
\begin{array}{l}
W(\pi \varepsilon^2) E'(\varepsilon) + \pi \varepsilon \varphi'(a) = \varphi'(a) a'(\varepsilon), \\
W'(\pi \varepsilon^2) E'(\varepsilon) + 2 \pi \varepsilon \varphi''(a)  = \frac{1}{2} \varphi''(a) a'(\varepsilon),
\end{array} \right.
\end{equation}
where we have used that
$$
U''(\pi \varepsilon^2;E(\varepsilon)) = U(\pi \varepsilon^2;E(\varepsilon)) - 3 U(\pi \varepsilon^2;E(\varepsilon))^5 =
\varphi(a) - 3 \varphi(a)^5 = \varphi''(a).
$$
Recall that $\varphi'(a) \neq 0$ for every $a \in (0,\infty)$.
If $\varphi''(a) \neq 0$ (which is true for every $a \in (0,\infty)$
except for $a = a_0 := \frac{1}{2} {\rm arccosh}(\sqrt{3})$),
then $E'(\varepsilon) \neq 0$ and the boundary conditions (\ref{BC-W}) are equivalent to
\be
\frac{2 W'(\pi \varepsilon^2)}{W(\pi \varepsilon^2)} = \frac{\varphi''(a) \left[ a'(\varepsilon) - 4 \pi \varepsilon \right]}{\varphi'(a) \left[ a'(\varepsilon) - \pi \varepsilon \right]},
\ee
which is incompatible with the required boundary condition (\ref{BC-balance}) for every $\varepsilon > 0$.
In the exceptional case $a = a_0$, for which $\varphi''(a_0) = 0$, it follows from (\ref{system-solution})
with $\gamma \neq 0$ that $W'(\pi \varepsilon^2) = 0$. However, differentiating the first-order invariant (\ref{first}) in $E$ yields
the relation
\begin{equation}
\label{tech-constraint}
1 = 2 U'(z) W'(z) - 2 U(z) W(z) \left[ 1 - 3 U(z)^4 \right], \quad z \in [-\pi \varepsilon^2, \pi \varepsilon^2].
\end{equation}
Since $W'(\pi \varepsilon^2) = 0$ and $1 - 3 U^4(\pi \varepsilon^2) = 0$ in the case $\varphi''(a_0) = 0$,
the constraint (\ref{tech-constraint}) yields a contradiction. Therefore, for every $\varepsilon \in (0,\infty)$, it is
impossible to satisfy the boundary conditions (\ref{system-solution}) for nonzero $(\alpha,\beta,\gamma)$,
which implies that $z(\mathcal{L}) = 0$.
\end{proof}

By Lemma \ref{lemma-3-1}, we have $z(\mathcal{L}) = 0$ for every $\omega < 0$.
By Courant's Min-Max theory,
it follows from the variational characterization (\ref{infB}) with a single constraint that $n(\mathcal{L}) \leq 1$.
Moreover, it follows from the exact computation (\ref{negative-result}) that $n(\mathcal{L}) \geq 1$,
hence $n(\mathcal{L}) = 1$ for every $\omega < 0$.

The assertion of Theorem \ref{theorem-degeneracy} is proven.

\section{Mass $\mu$ versus frequency $\omega$ for the standing waves}
\label{sec-mass}

Since $z(\mathcal{L}) = 0$ by Lemma \ref{lemma-3-1}, the 
self-adjoint linearized operator $\mathcal{L} : H^2_{\rm NK}(\mathcal{T}) \subset L^2(\mathcal{T}) \mapsto L^2(\mathcal{T})$ is one-to-one.
Since $\sigma_{\rm ac}(\mathcal{L}) = [|\omega|,\infty)$ with $|\omega| > 0$ by the same Lemma \ref{lemma-3-1}, $0$ is bounded away from $\sigma(\mathcal{L})$, so that there exists a positive constant $C$ such that 
$$
\| \mathcal{L} u \|_{L^2(\mathcal{T})} \geq C \| u \|_{L^2(\mathcal{T})}
$$
for every $u \in H^2_{\rm NK}(\mathcal{T})$. Hence, $\mathcal{L}$ is onto and there exists a bounded inverse operator $\mathcal{L}^{-1} : H^2_{\rm NK}(\mathcal{T}) \subset L^2(\mathcal{T})  \mapsto H^2_{\rm NK}(\mathcal{T}) \subset L^2(\mathcal{T})$. By using standard arguments based on the implicit function theorem, it follows 
that the map $(-\infty,0) \ni \omega \mapsto \Phi(\cdot,\omega) \in H^2_{\rm NK}(\mathcal{T})$
is $C^1$. Therefore, the mass $\mu=\mu(\omega) := Q(\Phi(\cdot,\omega))$ is a $C^1$ function of the frequency $\omega$.
This yields the first assertion of Theorem \ref{theorem-persistence}.

Next, we consider the asymptotic limits of $\mu(\omega)$ as $\omega \to 0$ and $\omega \to -\infty$
in order to prove the property (\ref{mass-property-1}) in Theorem \ref{theorem-persistence}.
This will be done separately using two different asymptotic methods.

The limit $\omega \to 0$ is handled by using the power series expansions.

\begin{lemma}
\label{lemma-4-1}
For small $\omega < 0$, we have
\begin{equation}
\label{mass-small-omega}
\mu(\omega) = \mu_{\RE^+} + 20 \pi^3 |\omega|^{3/2} + \mathcal{O}(|\omega|^{5/2}) > \mu_{\RE^+}.
\end{equation}
Moreover, there exists $\omega_2 \in (-\infty,0)$ such that $\mu'(\omega) < 0$ for $\omega \in (\omega_2,0)$.
\end{lemma}

\begin{proof}
The limit $\omega \to 0$ corresponds to the limit $\varepsilon \to 0$, for which
solutions of the boundary-value problem (\ref{sys1}) can be obtained by power series:
\begin{equation}
\label{power}
U(z) = U_+ \left[ 1 + \frac{1}{2} (1 - 3 U_+^4) z^2 + \frac{1}{24} (1 - 3 U_+^4) (1 - 15 U_+^4) z^4 + \mathcal{O}(z^6) \right],
\end{equation}
where $U_+ = U(0)$ is the same turning point as in the period function (\ref{period-function}). From the boundary conditions in (\ref{sys1}) we obtain
\begin{equation*}
\tanh(2a) = -\frac{\varphi'(a)}{\varphi(a)} = -\frac{2 U'(\pi \varepsilon^2)}{U(\pi \varepsilon^2)}
= 2 \pi \varepsilon^2 (3 U_+^4 - 1) \left[ 1 - \frac{1}{3} (1 + 3 U_+^4) \pi^2 \varepsilon^4 + \mathcal{O}(\varepsilon^8) \right]
\end{equation*}
and
\begin{equation*}
{\rm sech}(2a) = [U(\pm \pi \varepsilon^2)]^2 = U_+^2 \left[ 1 + (1 - 3 U_+^4) \pi^2 \varepsilon^4 +
\frac{1}{3} (1 - 3 U_+^4) (1 - 6 U_+^4) \pi^4 \varepsilon^8 + \mathcal{O}(\varepsilon^{12}) \right]
\end{equation*}
The two constraints can be written as the implicit 
equation $F(a,U_+;\varepsilon) = 0$ on the function $F(a,U_+;\varepsilon) : \mathbb{R}^2 \times \mathbb{R} \mapsto \mathbb{R}^2$. Thanks to the smoothness of $U, \varphi \in C^{\infty}$, we have $F \in C^{\infty}(\mathbb{R}^2 \times \mathbb{R})$. Moreover, 
$F(0,1;0) = 0$ and the Jacobian $D_{(a,U_+)} F(0,1;0)$ is invertible 
since $\det D_{(a,U_+)} F(0,1;0) = -4$.
By the Implicit Function Theorem for $C^{\infty}$ functions, 
for every small $\varepsilon$, there exists a unique solution 
$(a,U_+)$ of $F(a,U_+;\varepsilon) = 0$ near $(a,U_+) = (0,1)$; 
moreover, the dependence of $U_+$ and $a$ on $\varepsilon$ is $C^{\infty}$.
Solving the two nonlinear equations for $U_+$ and $a$ in terms of $\varepsilon$ with the power expansions yields the asymptotic solution:
\begin{equation}
\label{eq1a}
U_+ = 1 - 3 \pi^2 \varepsilon^4 + \mathcal{O}(\varepsilon^{8})
\end{equation}
and
\begin{equation}
\label{eq2a}
a = 2 \pi \varepsilon^2 - 28 \pi^3 \varepsilon^6 + \mathcal{O}(\varepsilon^{10}).
\end{equation}

We can now compute the mass $\mu(\omega)$ versus $\varepsilon$
as $\varepsilon \to 0$. We have
\begin{equation}
\label{mass-u-small}
\| u \|^2_{L^2(-\pi,\pi)} = 2 \int_0^{\pi \varepsilon^2} [U(z)]^2 dz =
2 \pi \varepsilon^2 - \frac{40}{3} \pi^3 \varepsilon^6 + \mathcal{O}(\varepsilon^{10})
\end{equation}
and
\begin{equation}
\label{mass-v-small}
\| v \|^2_{L^2(0,\infty)} = \int_0^{\infty} [V(z)]^2 dz = \arctan\left( e^{-2a} \right)
= \frac{\pi}{4} - 2 \pi \varepsilon^2  + \frac{100}{3} \pi^3 \varepsilon^6 + \mathcal{O}(\varepsilon^{10})
\end{equation}
so that
\begin{equation}
\label{mass-small}
\mu = \mu_{\RE^+} + 20 \pi^3 \varepsilon^6 + \mathcal{O}(\varepsilon^{10}).
\end{equation}
Since $\omega = -\varepsilon^4 < 0$, the asymptotic expansion (\ref{mass-small}) yields (\ref{mass-small-omega}).
The dependence of $\mu$ on $\varepsilon$ is $C^{\infty}$ and by the chain rule, we have
$\mu'(\omega) = -30 \pi^3 \varepsilon^2 + \mathcal{O}(\varepsilon^6)$ as $\varepsilon \to 0$.
This yields the assertion of the lemma.
\end{proof}

\begin{remark}
For the subcritical nonlinearities it was shown in \cite{[NPS15]}
that $\mu(\omega) \to 0$ as $\omega \to 0$ and $\mu'(\omega) < 0$ for small $|\omega|$. For the critical
nonlinearity, the leading order computation of $\mu'(\omega)$ was not conclusive as
$\omega \to 0$ in \cite{[NPS15]}. The power expansions above clarify this uncertainty and show
that $\mu'(\omega) < 0$ for small $|\omega|$.
\end{remark}

The limit $\omega \to -\infty$ is handled by using properties of elliptic functions.

\begin{lemma}
\label{lemma-5-1}
For large $\omega < 0$, we have
\begin{equation}
\label{mass-large-omega}
\mu(\omega) = \mu_{\RE} + \frac{8 \pi}{3}  |\omega|^{1/2} e^{-2\pi |\omega|^{1/2}} +
\mathcal{O}(e^{-2 \pi |\omega|^{1/2}}) > \mu_{\RE}.
\end{equation}
Moreover, there exists $\omega_1 \in (-\infty,\omega_2]$ such that $\mu'(\omega) > 0$ for $\omega \in (-\infty,\omega_1)$.
\end{lemma}

\begin{proof}
First, let us derive an exact solution of the quadrature (\ref{first})
with $E < 0$ given by (\ref{a-to-energy}). By using the variable $\rho := U^2$,
the first-order invariant (\ref{first}) is rewritten in the equivalent form:
\begin{equation}
\label{first-equivalent}
\frac{1}{4} (\rho')^2 = g_E(\rho) := E \rho + \rho^2 - \rho^4 = \rho (\rho_1 - \rho) (\rho_2 - \rho) (\rho_3 - \rho),
\end{equation}
where the nonzero roots $\rho_1$, $\rho_2$, and $\rho_3$ satisfy the constraints
\begin{equation}
\label{roots-rho}
\left\{ \begin{array}{l} \rho_1 + \rho_2 + \rho_3 = 0, \\
\rho_1 \rho_2 + \rho_1 \rho_3 + \rho_2 \rho_3 = -1, \\
\rho_1 \rho_2 \rho_3 = E. \end{array} \right.
\end{equation}
Since $g_E(0) = 0$ and $g_E'(0) = E < 0$, one root (say $\rho_3)$ is negative and the other two roots ($\rho_1$ and $\rho_2$) are either real and positive or complex-conjugate. Admissible solutions for $\rho = U^2 > 0$ 
exist only if the roots $\rho_1$ and $\rho_2$ are real and positive, 
so that we can order them as 
\[
\rho_3 < 0 < \rho_2 < \rho_1.
\]
Solving (\ref{roots-rho}) for $\rho_{1,2}$ and $E$ in terms of $\rho_3$ yields
\begin{equation}
\label{rho-1-2-3}
\rho_{1,2} = \frac{1}{2} |\rho_3| \pm \sqrt{1 - \frac{3}{4} \rho_3^2}, \quad
|E| = |\rho_3| (\rho_3^2 - 1).
\end{equation}
from which it follows that the roots $\rho_1$ and $\rho_2$ are real and positive if $|\rho_3| \in (1,\frac{2}{\sqrt{3}})$ which corresponds to $|E| \in (0,\frac{2}{3 \sqrt{3}})$.
It follows from (\ref{a-to-energy}) that $-\frac{2}{3 \sqrt{3}} < E < 0$ so that the roots $\rho_1$ and $\rho_2$ are real and positive for every $a \in (0,\infty)$.

Let us now write the explicit solution to the quadrature (\ref{first-equivalent}) in Jacobian elliptic functions ${\rm sn}$, ${\rm cn}$, and ${\rm dn}$ (see \cite{AS} for review). These elliptic functions are derived from 
the inversion of the elliptic integral of the first kind,
\begin{equation*}
x = F(\tau;k) = \int_{0}^{\tau} \frac{dt}{\sqrt{1-k^2\sin^2t}},
\end{equation*}
where $k\in(0,1)$ is the elliptic modulus. The complete elliptic integral 
is defined as $K(k) = F(\frac{\pi}{2};k)$. The first two Jacobi elliptic functions 
are defined by $\textrm{sn}(x;k) = \sin \tau$ and $\textrm{cn}(x;k) = \cos \tau$ such that  
\begin{equation}
\textrm{sn}^2(x;k) + \textrm{cn}^2(x;k)  = 1. 
\label{Jid1}
\end{equation}  
These functions  are smooth, sign-indefinite, and periodic with the period $4K(k)$. The third Jacobi elliptic function is defined from the quadratic formula  
\begin{equation}
\textrm{dn}^2(x;k) + k^2\textrm{sn}^2(x;k)  = 1. 
\label{Jid2}
\end{equation} 
The function ${\rm dn}(x;k)$ is given by the positive square root of (\ref{Jid2}), so that it is smooth, positive, and periodic with the period $2 K(k)$. The Jacobi elliptic functions are related by the derivatives:
\begin{equation} 
\label{der-Jac}
\left\{ \begin{array}{l}
\frac{d}{dx} \ \textrm{sn}(x;k) = \textrm{cn}(x;k) \ \textrm{dn}(x;k),\\
\frac{d}{dx} \ \textrm{cn}(x;k) = -\textrm{sn}(x;k) \ \textrm{dn}(x;k), \\ 
\frac{d}{dx} \ \textrm{dn}(x;k) = -k^2 \textrm{sn}(x;k) \ \textrm{cn}(x;k).
\end{array} \right.
\end{equation} 

As is well-known, see, e.g., \cite{ChenPel} for computational details,
the exact solution to the quadrature (\ref{first-equivalent}) can be written in the form:
\begin{eqnarray}
\label{exact-solution}
\rho(z) = \rho_3 + \frac{(\rho_1 - \rho_3) (\rho_2 - \rho_3)}{(\rho_2 - \rho_3) + (\rho_1 - \rho_2) {\rm sn}^2(\nu z;k)}
= \frac{\rho_1 |\rho_3| {\rm dn}^2(\nu z;k)}{\rho_1 + |\rho_3| - \rho_1 {\rm dn}^2(\nu z;k)},
\end{eqnarray}
where
$$
\nu = \sqrt{\rho_1 ( \rho_2 - \rho_3)}, \quad k = \frac{\sqrt{|\rho_3| (\rho_1 - \rho_2)}}{\sqrt{\rho_1 ( \rho_2 - \rho_3)}}.
$$
The solution exists in $[\rho_2,\rho_1]$, hence $\rho(z) > 0$ for every $z$. In addition, $\rho(0) = \rho_1$ and $\rho(L/2) = \rho_2$,
where $L = 2 K(k)/\nu$ is the period of the exact periodic solution.

We shall now explore the asymptotic limit $\varepsilon \to \infty$, which corresponds to
the limit $a \to \infty$. It follows from
(\ref{a-to-energy}) in the limit $a \to \infty$ that
\[
E = -\frac{3}{2} e^{-2a} + \mathcal{O}(e^{-6a}).
\]
By solving the cubic equation for $|\rho_3|$ in (\ref{rho-1-2-3})
and using the explicit expressions for $\rho_{1,2}$, we obtain
in the same limit:
\begin{equation}
\label{expansions-rho}
\left\{
\begin{array}{l}
\rho_1 = 1 - \frac{3}{4} e^{-2a} + \mathcal{O}(e^{-4a}), \\
\rho_2 = \frac{3}{2} e^{-2a} + \mathcal{O}(e^{-4a}), \\
|\rho_3| = 1 + \frac{3}{4} e^{-2a} + \mathcal{O}(e^{-4a}),
\end{array} \right.
\end{equation}
from which we obtain
\begin{equation}
\label{expansion-k}
\nu = 1 + \frac{3}{4} e^{-2a} + \mathcal{O}(e^{-4a}), \quad k  = 1 - \frac{3}{2} e^{-2a} + \mathcal{O}(e^{-4a}).
\end{equation}
Approximations of elliptic functions in terms of hyperbolic functions (see 16.15 in \cite{AS}) were justified in
Proposition 4.6 and Appendix A in \cite{MarPel}. In the limit $k \to 1$ and $x \to \infty$ such that
$|e^{x} (1-k)| \leq C e^{-x}$ for a given $(k,x)$-independent positive constant $C$, the elliptic functions satisfy the expansions:
\begin{eqnarray}
\label{expansion-dn}
{\rm dn}(x;k) & = & 2 e^{-x} + \frac{1}{4} e^{x} (1 - k) + \mathcal{O}((1-k) |x| e^{-x}), \\
\label{expansion-cn}
{\rm cn}(x;k) & = & 2 e^{-x} - \frac{1}{4} e^{x} (1 - k) + \mathcal{O}((1-k) |x| e^{-x}),
\end{eqnarray}
and
\begin{eqnarray}
\label{expansion-sn}
{\rm sn}(x;k) = 1 + \mathcal{O}(1-k).
\end{eqnarray}
It follows from (\ref{exact-solution}), (\ref{expansions-rho}), (\ref{expansion-k}), and (\ref{expansion-dn}) that
$$
U(\pi \varepsilon^2) = \sqrt{\rho(\pi \varepsilon^2)} = \frac{1}{\sqrt{2}} \left[ 2 e^{-\pi \varepsilon^2}
+ \frac{1}{4} (1-k) e^{\pi \varepsilon^2} + \mathcal{O}((1-k) \varepsilon^2 e^{-\pi \varepsilon^2}) \right],
$$
as long as $1 - k = \mathcal{O}(e^{-2\pi \varepsilon^2})$ and $a = \mathcal{O}(\varepsilon^2)$ as $\varepsilon \to \infty$. It follows from 
(\ref{der-Jac}), (\ref{exact-solution}), (\ref{expansions-rho}),
(\ref{expansion-k}), (\ref{expansion-cn}), and (\ref{expansion-sn}) that
$$
U'(\pi \varepsilon^2) = -\frac{1}{\sqrt{2}} \left[ 2 e^{-\pi \varepsilon^2}
- \frac{1}{4} (1-k) e^{\pi \varepsilon^2} + \mathcal{O}((1-k) \varepsilon^2 e^{-\pi \varepsilon^2}) \right].
$$
Since the boundary conditions in (\ref{sys1}) yield
$$
2U'(\pi \varepsilon^2) = -\tanh(2a) U(\pi \varepsilon^2),
$$
we obtain the following implicit equation for $k$:
$$
2 \left[ 2 e^{-\pi \varepsilon^2}
- \frac{1}{4} (1-k) e^{\pi \varepsilon^2} + \mathcal{O}((1-k) \varepsilon^2 e^{-\pi \varepsilon^2}) \right] =
\left[ 2 e^{-\pi \varepsilon^2}
+ \frac{1}{4} (1-k) e^{\pi \varepsilon^2} + \mathcal{O}((1-k) \varepsilon^2 e^{-\pi \varepsilon^2}) \right],
$$
as long as $1 - k = \mathcal{O}(e^{-2\pi \varepsilon^2})$ and $a = \mathcal{O}(\varepsilon^2)$
as $\varepsilon \to \infty$. 
After multiplying by $e^{-\pi \varepsilon^2}$ and simplifying similar terms, 
we rewrite the implicit equation in the form:
\begin{equation}
\label{impl-eq-k}
\frac{3}{4} (1-k) - 2 e^{-2\pi \varepsilon^2} + \mathcal{O}((1-k) \varepsilon^2 e^{-2 \pi \varepsilon^2}) = 0.
\end{equation}
Thanks to the smoothness of $U, \varphi \in C^{\infty}$, coefficients of this implicit equation are $C^{\infty}$ in $k$ and $\varepsilon$. There exists a root $k = 1$ as $\varepsilon \to \infty$; moreover, the root is simple. 
By the Implicit Function Theorem for $C^{\infty}$ functions, 
for every large $\varepsilon$, there exists a unique solution 
$k$ of the implicit equation (\ref{impl-eq-k})  near $k = 1$; 
moreover, the dependence of $k$ on $\varepsilon$ is $C^{\infty}$.
The asymptotic expansion of the simple root of $k$ is given by 
\begin{equation}
\label{k-unique}
k = 1 - \frac{8}{3} e^{-2\pi \varepsilon^2} + \mathcal{O}(\varepsilon^2 e^{-4 \pi \varepsilon^2}).
\end{equation}
In combination with the expansion for $k$ in (\ref{expansion-k}), this yields the unique asymptotic
balance at
\begin{equation}
\label{choices-unique}
e^{-2a} = \frac{16}{9} e^{-2\pi\varepsilon^2} + \mathcal{O}(\varepsilon^2 e^{-4 \pi \varepsilon^2}),
\end{equation}
or equivalently,
\begin{equation}
\label{a-unique-root}
a = \pi \varepsilon^2 + \log(\frac{3}{4}) + \mathcal{O}(\varepsilon^2 e^{-2\pi \varepsilon^2}).
\end{equation}
The dependence of $a$ on $\varepsilon$ is $C^{\infty}$. 
This completes the asymptotic construction of the solution (\ref{exact-solution})
in the limit $\varepsilon \to \infty$. 

We can now compute the mass $\mu(\omega)$ given by (\ref{mass}) versus $\varepsilon$ as $\varepsilon \to \infty$.
As it is explained in Appendix \ref{app-asymptotics}, we obtain
\begin{eqnarray}
\label{int-mass}
\| u \|^2_{L^2(-\pi,\pi)} = 2 \int_0^{\pi \varepsilon^2} \rho(z) dz
= \frac{\pi}{2} + \frac{8 \pi}{3}  \varepsilon^2 e^{-2\pi \varepsilon^2} + \mathcal{O}(e^{-2\pi \varepsilon^2}).
\end{eqnarray}
On the other hand, thanks to the asymptotic balance (\ref{choices-unique}) we have
\begin{equation*}
\| v \|^2_{L^2(0,\infty)} = \arctan\left( e^{-2a} \right) = \mathcal{O}(e^{-2 \pi \varepsilon^2}),
\end{equation*}
so that
\begin{equation}
\label{mass-large}
\mu = \mu_{\RE} + \frac{8 \pi}{3}  \varepsilon^2 e^{-2\pi \varepsilon^2} + \mathcal{O}(e^{-2 \pi \varepsilon^2}).
\end{equation}
Since $\omega = -\varepsilon^4 < 0$, the asymptotic expansion (\ref{mass-large}) yields (\ref{mass-large-omega}).
The dependence of $\mu$ on $\varepsilon$ is $C^{\infty}$ and by the chain rule, we have
$\mu'(\omega) = \frac{8 \pi^2}{3} e^{-2\pi \varepsilon^2} + \mathcal{O}(\varepsilon^{-2} e^{-2\pi \varepsilon^2})$
as $\varepsilon \to \infty$. This yields the assertion of the lemma.
\end{proof}

\begin{figure}[h!]
	\centering
	\includegraphics[width=7.5cm,height=7cm]{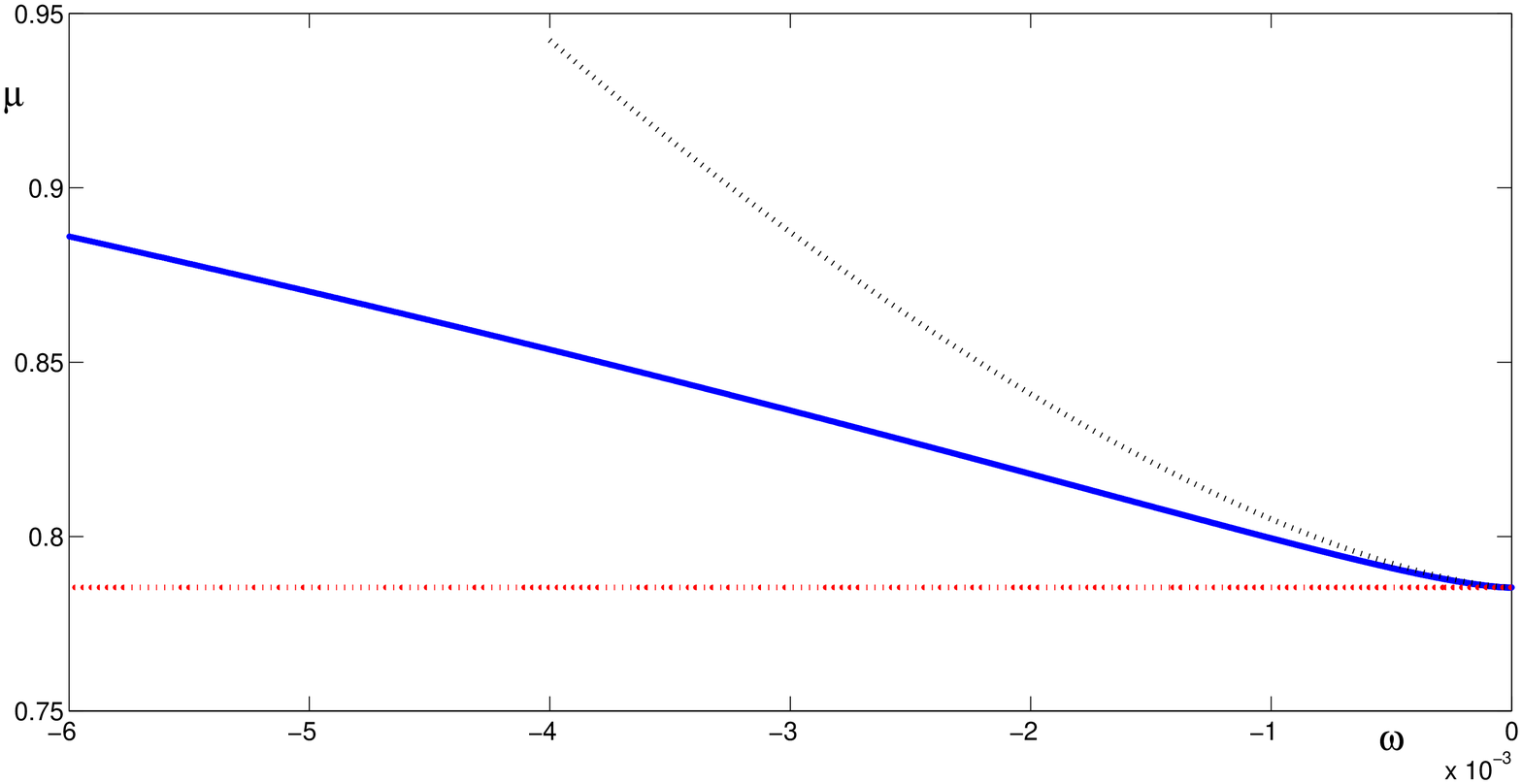}
	\includegraphics[width=7.5cm,height=7cm]{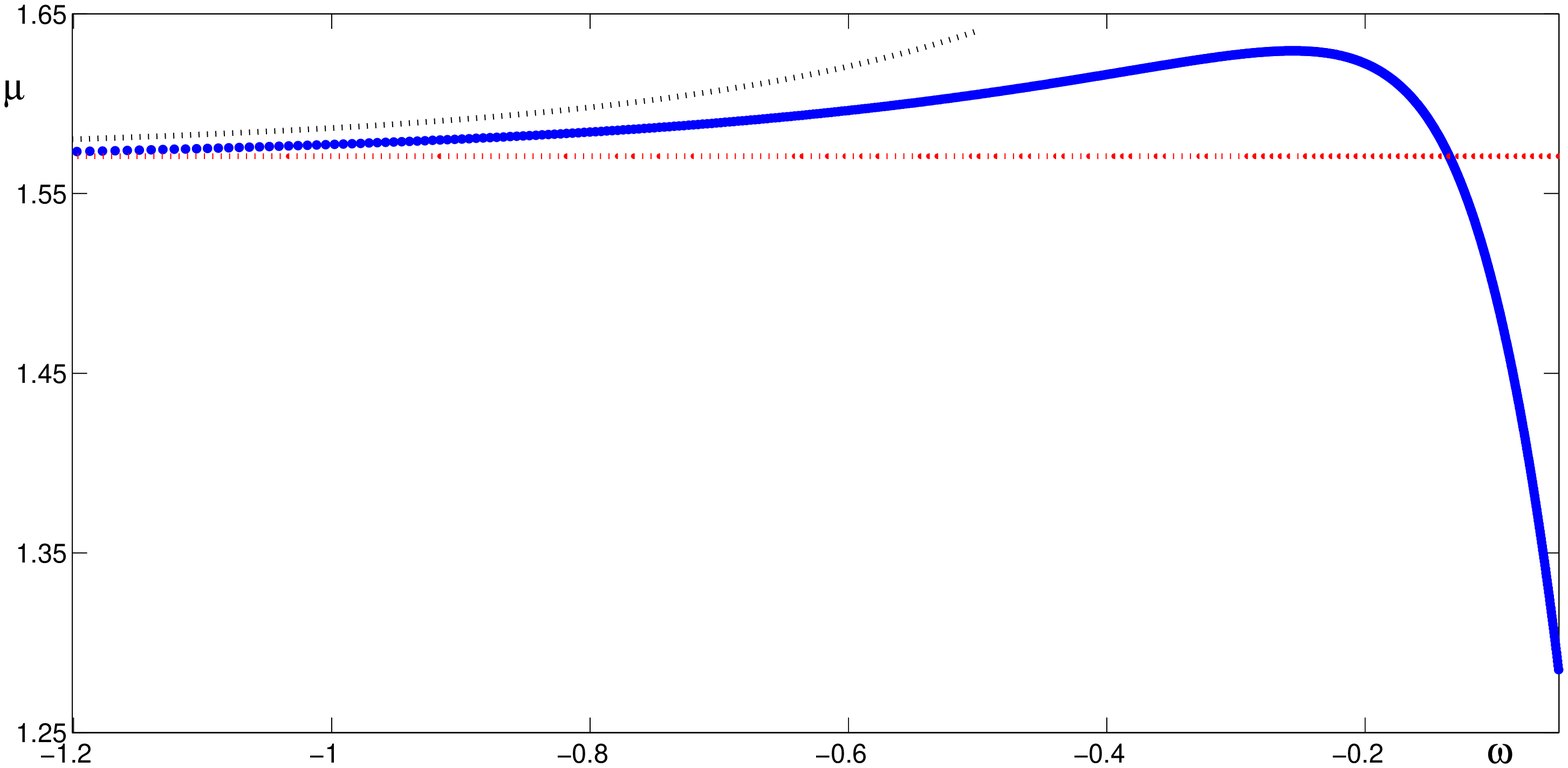}
	\caption{Asymptotics of mass $\mu$ versus $\omega$ for $\omega \to 0$ (left) and $\omega \to -\infty$ (right).
Dashed lines show the levels (\ref{mass-half-soliton}) and (\ref{mass-soliton}), whereas the solid lines
show the asymptotic expressions (\ref{mass-small-omega}) and (\ref{mass-large-omega}).}
	\label{fig-mass-small-large}
\end{figure}

Let us illustrate numerically the implicit solution defined by the quadrature (\ref{quadrature}).
For each fixed $U_0 \in (0,1)$, we find $U_+ \in [U_*,1)$, where $U_* := \frac{1}{3^{1/4}}$,
from numerical solution of $E + U^2 - U^6 = 0$
with $E$ given by (\ref{a-to-energy}). Then, we integrate the quadrature (\ref{quadrature})
numerically, hence obtaining a unique value of $\varepsilon^2 = |\omega|^{1/2}$
for each $U_0$. Then, we compute the mass $\mu$ from the following integral:
\begin{equation}
\label{mass-representation}
\mu = 2 \int_{U_0}^{U_+} \frac{u^2 du}{\sqrt{E + u^2 - u^6}} + {\rm arctan}(e^{-2a}),
\end{equation}
where $a \in (0,\infty)$ is expressed from $U_0 \in (0,1)$ by the explicit formula:
\begin{equation}
\label{a-representation}
e^{2a} = \frac{1 + \sqrt{1 - U_0^4}}{U_0^2}.
\end{equation}
By using the numerical integration above, we have obtained the mapping $\omega \mapsto \mu(\omega)$,
which is plotted on Figure \ref{fig-mass}. Figure \ref{fig-mass-small-large} shows
the asymptotic dependencies (\ref{mass-small-omega}) and (\ref{mass-large-omega}) by solid lines
superposed together with the numerical data for $\mu(\omega)$ by black dots.
The levels (\ref{mass-half-soliton}) and (\ref{mass-soliton}) are shown by dotted lines.

Finally, we prove the monotonicity of the mapping $\omega \mapsto \mu(\omega)$ given by the property (\ref{mass-property-2})
in Theorem \ref{theorem-persistence}. The following lemma gives the result.

\begin{lemma}
\label{lemma-mass-monotonicity}
There exists a unique $\omega_1 \in (-\infty,0)$ such that
$\mu'(\omega) > 0$ for $\omega \in (-\infty,\omega_1)$ and $\mu'(\omega) < 0$ for $\omega \in (\omega_1,0)$.
\end{lemma}

\begin{proof}
We recall the representation of the mass $\mu(\omega)$ in the form (\ref{mass-representation}),
where $U_0 \in (0,1)$ is the only parameter, whereas $a$ is given by (\ref{a-representation})
and $\omega = -\varepsilon^4$ is uniquely determined from $\pi \varepsilon^2 = T(U_0)$
by the period function (\ref{period-function}). By Lemma \ref{lemma-2-1}, the
map $(0,\infty) \ni a \mapsto \varepsilon(a) \in (0,\infty)$ is monotonically increasing. Since the maps
$(0,\infty) \ni \varepsilon \mapsto \omega(\varepsilon) \in (-\infty,0)$
and $(0,1) \ni U_0 \mapsto a(U_0) \in (0,\infty)$ are monotonically decreasing, monotonicity of the map
$(-\infty,0) \ni \omega \mapsto \mu(\omega) \in (0,\infty)$
is identical to the monotonicity of the map $(0,1) \ni U_0 \mapsto \mu(U_0) \in (0,\infty)$.
Let us define
\begin{equation}
\label{mass-representation2}
B(U_0) := \int_{U_0}^{U_+} \frac{u^2 du}{\sqrt{E + A(u)}},
\end{equation}
so that $\mu = 2 B(U_0) + {\rm arctan}(e^{-2a})$ according to (\ref{mass-representation}).
Here we remind that $A(u) = u^2 - u^6$, the value of $E$ is given by $E = -\frac{3}{4} A(U_0)$,
and $U_+$ is the largest positive root of $E + A(u) = 0$ such that $U_+ \geq U_* := \frac{1}{3^{1/4}}$.

Define $p := \sqrt{E + A(u)}$ and compute for every $u \in (0,1)$:
\begin{eqnarray*}
d \left( \frac{2 p u^2 [A(u) - A(U_*)]}{A'(u)} \right) & = &
2 \left[ 1 + \frac{2 (1+9u^4) [A(u) - A(U_*)]}{[A'(u)]^2} \right] p u^2 du + \frac{u^2 [A(u) - A(U_*)]}{p} du,
\end{eqnarray*}
where we have used explicitly $A'(u) = 2u (1 - 3 u^4)$ and $A''(u) = 2 (1 - 15 u^4)$.
All terms in this expression are non-singular for every $u \in (0,1)$. It enables us
to express the function $B(U_0)$ in the equivalent way:
\begin{eqnarray*}
[E + A(U_*)] B(U_0) & = & \int_{U_0}^{U_+} p u^2 du - \int_{U_0}^{U_+} \frac{u^2 [A(u) - A(U_*)]}{p} du \\
& = & \int_{U_0}^{U_+} \left[ 3 + \frac{4 (1 + 9 u^2) [A(u) - A(U_*)]}{[A'(u)]^2} \right] p u^2 du \\
& \phantom{t} & + \frac{U_0^2 [A(U_0) - A(U_*)]}{A'(U_0)} \sqrt{A(U_0)},
\end{eqnarray*}
where we have used that $2 \sqrt{E + A(U_0)} = \sqrt{A(U_0)}$. The right-hand side is $C^1$ in $U_0$ on $(0,1)$,
hence the derivative is computed explicitly in the form
\begin{eqnarray}
[E + A(U_*)] B'(U_0) = \frac{1}{4 \sqrt{3}} A'(U_0) \int_{U_0}^{U_+} \frac{(1 - \sqrt{3} u^2)}{(1 + \sqrt{3} u^2)^2} \frac{du}{p} - \frac{U_0^2 A(U_*)}{2\sqrt{A(U_0)}},
\label{derivative-B}
\end{eqnarray}
where we have used again the explicit representation for $A(u)$ and $E$.
Similarly, we compute directly with the help of (\ref{a-representation}) that
\begin{equation}
\label{derivative-tail}
\frac{d}{dU_0} {\rm arctan}(e^{-2a}) = \frac{U_0}{\sqrt{1-U_0^4}} = \frac{U_0^2}{\sqrt{A(U_0)}}.
\end{equation}
Bringing (\ref{derivative-B}) and (\ref{derivative-tail}) together yields
\begin{equation}
\label{derivative-mass}
[E + A(U_*)] \mu'(U_0) = \frac{1}{2 \sqrt{3}} A'(U_0) \int_{U_0}^{U_+} \frac{(1 - \sqrt{3} u^2)}{(1 + \sqrt{3} u^2)^2} \frac{du}{p}
- \frac{3}{4} U_0^2 \sqrt{A(U_0)}.
\end{equation}

Next, we show that $\mu'(U_0) < 0$ if $U_0 \in [U_*,1)$. Since $A'(U_0) \leq 0$ for $U_0 \in [U_*,1)$ and
$1 - \sqrt{3} u^2 \leq 0$ for $u \in [U_*,1)$, the first term in (\ref{derivative-mass})
is positive, whereas the second term is negative. In order to combine them together, we integrate by
parts and obtain:
\begin{eqnarray*}
\int_{U_0}^{U_+} \frac{(1 - \sqrt{3} u^2)}{(1 + \sqrt{3} u^2)^2} \frac{du}{\sqrt{E + A(u)}}
& = & \frac{1}{2} \int_{U_0}^{U_+} \frac{A'(u)}{u (1 + \sqrt{3} u^2)^3} \frac{du}{\sqrt{E + A(u)}} \\
& = & - \frac{\sqrt{A(U_0)}}{2 U_0 (1 + \sqrt{3} U_0^2)^3} + \int_{U_0}^{U_+} \frac{(1 + 7 \sqrt{3} u^2) \sqrt{E + A(u)}}{u^2 (1 + \sqrt{3} u^2)^4} du.
\end{eqnarray*}
Substituting this representation into (\ref{derivative-mass}) yields
\begin{eqnarray*}
[E + A(U_*)] \mu'(U_0) & = & \frac{1}{2 \sqrt{3}} A'(U_0) \int_{U_0}^{U_+} \frac{(1 + 7 \sqrt{3} u^2) \sqrt{E + A(u)}}{u^2 (1 + \sqrt{3} u^2)^4} du \\
& \phantom{t} & - \frac{3}{4} U_0^2 \sqrt{A(U_0)} - \frac{A'(U_0)}{4 \sqrt{3} U_0 (1 + \sqrt{3} U_0^2)^3} \sqrt{A(U_0)}
\end{eqnarray*}
The first term in the right-hand side is now negative since $A'(U_0) \leq 0$ for $U_0 \in [U_*,1)$, whereas the other two terms are combined together
to give a negative expression:
\begin{eqnarray*}
- \frac{3}{4} U_0^2 \sqrt{A(U_0)} - \frac{A'(U_0)}{4 \sqrt{3} U_0 (1 + \sqrt{3} U_0^2)^3} \sqrt{A(U_0)}
= - \frac{2 + \sqrt{3} U_0^2 + 18 U_0^4 + 9 \sqrt{3} U_0^6}{4 \sqrt{3} (1 + \sqrt{3} U_0^2)^2} \sqrt{A(U_0)}.
\end{eqnarray*}
Hence $\mu'(U_0) < 0$ if $U_0 \in [U_*,1)$.

Finally, we show that there exists a unique $U_1 \in (0,U_*)$ such that $\mu'(U_1) = 0$.
Consider all possible values of $U_0 \in (0,U_*)$ for which $\mu'(U_0) = 0$. It follows
from (\ref{derivative-mass}) that this value of $U_0 \in (0,U_*)$ is a solution of the nonlinear equation
\begin{equation}
\label{root-F-G}
F(U_0) := \int_{U_0}^{U_+} \frac{(1 - \sqrt{3} u^2)}{(1 + \sqrt{3} u^2)^2} \frac{du}{\sqrt{E + A(u)}} =
\frac{3 \sqrt{3} U_0^2 \sqrt{1 - U_0^4}}{4 (1 - 3 U_0^4)} =: G(U_0).
\end{equation}
The map $(0,U_*) \ni U_0 \mapsto G(U_0) \in (0,\infty)$ is monotonically increasing due to the following computation:
$$
G'(U_0) = \frac{3 \sqrt{3} U_0 (1 + U_0^4)}{2 (1 - 3 U_0^4)^2 \sqrt{1 - U_0^4}} > 0
$$
and the limits
$$
\lim_{U_0 \to 0} G(U_0) = 0 \quad \mbox{\rm and} \quad \lim_{U_0 \to U_*} G(U_0) = \infty.
$$
On the other hand, the map $(0,U_*) \ni U_0 \mapsto F(U) \in \mathbb{R}$ is monotonically decreasing.
Indeed, by using the same integration by parts as above, we write
$$
F(U_0) =  \int_{U_0}^{U_+} \frac{(1 + 7 \sqrt{3} u^2) \sqrt{E + A(u)}}{u^2 (1 + \sqrt{3} u^2)^4} du
- \frac{\sqrt{A(U_0)}}{2 U_0 (1 + \sqrt{3} U_0^2)^3},
$$
where the first integral can be differentiated in $U_0$. Thus, we obtain
\begin{eqnarray*}
F'(U_0) & = & -\frac{3}{8} A'(U_0) \int_{U_0}^{U_+} \frac{(1 + 7 \sqrt{3} u^2)}{u^2 (1 + \sqrt{3} u^2)^4} \frac{du}{\sqrt{E + A(u)}} \\
& \phantom{t} & - \frac{(1 + 7 \sqrt{3} U_0^2) \sqrt{1 - U_0^4}}{2 U_0 (1 + \sqrt{3} U_0^2)^4}
+ \frac{U_0^3}{(1 + \sqrt{3} U_0^2)^3 \sqrt{1 - U_0^4}} + \frac{3 \sqrt{3} U_0 \sqrt{1 - U_0^4}}{(1 + \sqrt{3} U_0^2)^4},
\end{eqnarray*}
where the first term is negative since $A'(U_0) > 0$ for $U_0 \in (0,U_*)$. We check that the other terms
are combined together in the negative expression for $U_0 \in (0,U_*)$:
\begin{eqnarray*}
& \phantom{t} &
- \frac{(1 + 7 \sqrt{3} U_0^2) \sqrt{1 - U_0^4}}{2 U_0 (1 + \sqrt{3} U_0^2)^4}
+ \frac{U_0^3}{(1 + \sqrt{3} U_0^2)^3 \sqrt{1 - U_0^4}} + \frac{3 \sqrt{3} U_0 \sqrt{1 - U_0^4}}{(1 + \sqrt{3} U_0^2)^4} \\
& = &
- \frac{\sqrt{1 - U_0^4}}{2 U_0 (1 + \sqrt{3} U_0^2)^3}
+ \frac{U_0^3}{(1 + \sqrt{3} U_0^2)^3 \sqrt{1 - U_0^4}} \\
& = & - \frac{1 - \sqrt{3} U_0^2}{2 U_0 (1 + \sqrt{3} U_0^2)^2 \sqrt{1 - U_0^4}}.
\end{eqnarray*}
Hence $F'(U_0) < 0$ for $U_0 \in (0,U_*)$. It is clear that
$$
\lim_{U_0 \to 0} F(U_0) = \infty \quad \mbox{\rm and} \quad
\lim_{U_0 \to U_*} F(U_0) = F(U_*) < 0.
$$
By monotonicity and range of $F$ and $G$, there exists a unique
$U_1 \in (0,U_*)$ for which $F(U_1) = G(U_1)$. Moreover, $U_1$ is a simple root of the nonlinear equation (\ref{root-F-G}).
Hence $\mu'(U_0) > 0$ for $U_0 \in (0,U_1)$ and $\mu'(U_0) < 0$ for $U_0 \in (U_1,1)$.
\end{proof}

\appendix

\section{Characterization of the spectrum of $\Delta$ in $L^2(\mathcal{T})$}
\label{app-spectrum}

For completeness, we include the well-known characterization of $\sigma(-\Delta)$ in $L^2(\mathcal{T})$,
as in the following proposition.

\begin{proposition}
\label{deltaspectrum}
The spectrum of $-\Delta$ in $L^2(\mathcal{T})$ is given by $\sigma(-\Delta) = [0,\infty)$
and consists of the absolutely continuous spectrum $\sigma_{ac}(-\Delta) = [0,\infty)$ and
a sequence of simple embedded eigenvalues $\{ n^2 \}_{n \in \mathbb{N}}$.
\end{proposition}

\begin{proof}
Let us consider the spectral problem $-\Delta U = \lambda U$ with $U = (u,v) \in H^2_{\rm NK}(\mathcal{T})$.
Due to the geometry of the tadpole graph $\mathcal{T}$, the spectrum of $-\Delta$ in $L^2(\mathcal{T})$
is the union of two sets: the set of $\lambda$ for which $v = 0$  and the set of $\lambda$ for which $v \neq 0$.

The first set is defined by the pure point
spectrum of the spectral problem
\begin{equation}
\label{embedded-spectrum}
\left\{ \begin{array}{ll} - u'' = \lambda u, \quad & x \in (-\pi,\pi), \\
u(-\pi) = u(\pi) = 0, & \\
u'(-\pi) = u'(\pi). & \end{array} \right.
\end{equation}
Eigenvalues of the spectral problem (\ref{embedded-spectrum}) are located
at $\{0, 1, 4, 9, \dots \}$ and for each $\lambda = n^2$, $n \in \NA$, the eigenfunction
of $-\Delta$ is given by
\begin{equation}
\label{eigenfunction-1}
\left\{ \begin{array}{ll}
u(x) = \sin(nx), \quad & x \in [-\pi,\pi], \\
v(x) = 0, \quad & x \in [0,\infty). \end{array} \right.
\end{equation}

The second set includes the absolute continuous spectrum of $-\Delta$ located on $[0,\infty)$
and for each $\lambda = k^2$ with $k \in [0,\infty)$ the Jost function of $-\Delta$ is given by
\begin{equation}
\label{eigenfunction-2}
\left\{ \begin{array}{ll}
u(x) = a(k) \left[ e^{ikx} + e^{-ikx} \right], \quad & x \in [-\pi,\pi], \\
v(x) = e^{ikx} + b(k) e^{-ikx}, \quad & x \in [0,\infty). \end{array} \right.
\end{equation}
where
\begin{equation}
\label{a-b-expressions}
a(k) = \frac{1}{\cos(\pi k) + 2i \sin(\pi k)}, \quad b(k) = \frac{\cos(\pi k) - 2 i \sin(\pi k)}{\cos(\pi k) + 2i \sin(\pi k)},
\end{equation}
are found from the Neumann--Kirchhoff boundary conditions (\ref{dd}).
Because $a(k)$ and $b(k)$ are bounded and nonzero, there are no spectral singularities
in the absolute continuous spectrum of $-\Delta$ in $L^2(\mathcal{T})$.

It remains to check if the second set includes isolated eigenvalues $\lambda < 0$ with $v \neq 0$.
Representing a possible eigenfunction of $-\Delta$ for $\lambda < 0$ as
\begin{equation}
\left\{ \begin{array}{ll} u(x) = \frac{\cosh(\sqrt{|\lambda|} x)}{\cosh(\pi \sqrt{|\lambda|})}, \quad & x \in [-\pi,\pi], \\
v(x) = e^{-\sqrt{|\lambda|} x}, \quad & x \in [0,\infty), \end{array} \right.
\end{equation}
we obtain from the Neumann--Kirchhoff boundary conditions (\ref{dd}) that $\lambda$ is a solution to the transcendental equation:
\begin{equation}
1 + 2 \tanh(\pi \sqrt{|\lambda|}) = 0,
\end{equation}
which has no roots for real $\lambda$.
\end{proof}

\begin{remark}
It follows from (\ref{a-b-expressions}) that $a(k)$ and $b(k)$ are free of singularities for every $k \in [0,\infty)$ including the values
$k = n$, $n \in \mathbb{N}$ which correspond to the embedded eigenvalues. This is because
the odd subspace of eigenfunctions (\ref{eigenfunction-1}) for embedded eigenvalues and
the even subspace of Jost functions (\ref{eigenfunction-2}) for the absolute continuous spectrum
are uncoupled in the Neumann--Kirchhoff boundary conditions (\ref{dd}).
\end{remark}

\section{Relation between the constrained minimization problems}
\label{app-relation}

The constrained minimization problem (\ref{infB}) is related to the minimization of the action
\begin{equation}
S_{\omega}(U) = E(U) - \omega Q(U)
\end{equation}
on the Nehari manifold
\begin{equation}
B_{\omega}(U) - 3 \| U \|^6_{L^6(\mathcal{T})}=0,
\end{equation}
which characterizes the set of solutions of the stationary NLS equation (\ref{eqstaz}).
The following proposition establishes a relation of the latter minimization problem
with minimization of $B_{\omega}(U)$ at fixed $\| U \|_{L^6(\mathcal{T})}^6$.
Note that this result is not used in the main part of our paper and is added here for completeness.

\begin{proposition}
\label{prop-Appendix-B}
For every $\omega < 0$, there exists $\mathcal{M}(\omega) > 0$ such that
$\mathcal{P}(\omega)\leq\mathcal{N}(\omega)$, where $\mathcal{P}(\omega)$
and $\mathcal{N}(\omega)$ represent the following two constrained minimization problems:
\begin{equation}
\label{Nehari}
\mathcal{N}(\omega) := \inf_{U \in H^1_{\rm C}(\mathcal{T}) \backslash \{0\}} \left\{ \frac{3}{2} S_{\omega}(U) : \;\; \| U \|^6_{L^6(\mathcal{T})}=\frac{1}{3}B_{\omega}(U)  \right\}
\end{equation}
and
\begin{equation}
\label{L6norm}
\mathcal{P}(\omega) := \inf_{U \in H^1_{\rm C}(\mathcal{T})} \left\{ B_{\omega}(U)\ :\  \| U \|^6_{L^6(\mathcal{T})}  = \mathcal{M}^6(\omega)  \right\}
\end{equation}
Moreover, if $U$ is a minimizer of the variational problem \eqref{L6norm}, then
there exist $\beta(\omega) \geq 1$ such that $V :=\beta(\omega)U $ is a critical point of the variational problem \eqref{Nehari},
whereas if $V$ is a minimizer of the variational problem \eqref{Nehari}, then
there exists $\beta(\omega) \geq 1$ such that $U := V/\beta(\omega)$ is a critical point of the variational problem \eqref{L6norm}.
\end{proposition}

\begin{proof}
By using the constraint $B_{\omega}(U) = 3 \| U \|^6_{L^6(\mathcal{T})}$, we rewrite (\ref{Nehari}) in the equivalent form:
\begin{eqnarray}
\label{problem-1}
\mathcal{N}(\omega) = \inf_{U \in H^1_{\rm C}(\mathcal{T}) \backslash \{0\}} \left\{ B_{\omega}(U)\ : \| U \|^6_{L^6(\mathcal{T})} =\frac{1}{3} B_{\omega}(U) \right\}.
\end{eqnarray}
It follows from the lower bound in (\ref{B-equivalence}) and the Sobolev's embedding (\ref{Sobolev-embedding}) that
$$
 \| U \|^6_{L^6(\mathcal{T})} =\frac{1}{3}B_{\omega}(U)\geq \frac{C_-(\omega)}{3C^2} \| U \|^2_{L^6(\mathcal{T})}.
$$
Since $U \neq 0$, this implies that
\beq\label{Momega}
 \| U \|_{L^6(\mathcal{T})} \geq \left( \frac{C_-(\omega)}{3 C^2} \right)^{1/4} =: \mathcal{M}(\omega).
\eeq
Hence, $\mathcal{N}'(\omega) \leq \mathcal{N}(\omega)$, where
\begin{eqnarray}
\label{problem-2}
\mathcal{N}'(\omega) :=  \inf_{U \in H^1_{\rm C}(\mathcal{T})} \left\{ B_{\omega}(U)\ :\  \| U \|^6_{L^6(\mathcal{T})}  \geq \mathcal{M}^6(\omega)  \right\}.
\end{eqnarray}
where $\mathcal{M}(\omega)$ is uniquely defined by \eqref{Momega}.

By comparing \eqref{L6norm} and \eqref{problem-2}, it is obvious that  $\mathcal{N}'(\omega) \leq \mathcal{P(\omega)}$.
In order to show that $\mathcal{N}'(\omega) = \mathcal{P(\omega)}$, we will show the
reverse inequality $\mathcal{N}'(\omega)\geq \mathcal{P}(\omega)$. To do so, we
let $\{ U_n \}_{n\in \mathbb{N}} \in H^1_{\rm C}(\mathcal{T})$ be a minimizing sequence for the variational problem (\ref{problem-2})
satisfying $ \| U_n \|^6_{L^6(\mathcal{T})}\geq \mathcal{M}(\omega)$ for every $n \in \mathbb{N}$. Then,
we have $ \| \alpha_n U_n \|^6_{L^6(\mathcal{T})}=\mathcal{M}(\omega)$ with
$$
\alpha_n := \frac{\mathcal{M}(\omega)}{ \| U_n \|_{L^6(\mathcal{T})}} \leq 1, \quad n \in \mathbb{N}.
$$
Since $\alpha_n \leq 1$ for every $n \in \mathbb{N}$, this implies that
$$
\mathcal{P}(\omega) \leq B_{\omega}(\alpha_n U_n)=\alpha^2_n B_{\omega}( U_n )\leq B_{\omega}( U_n )$$
Taking the limit $n \to \infty$ yields $\mathcal{P}(\omega) \leq \mathcal{N}'(\omega)$ and so
$\mathcal{P}(\omega) = \mathcal{N}'(\omega) \leq \mathcal{N}(\omega)$.

Euler-Lagrange equation for the variational problem \eqref{L6norm} is given in the weak form by
\beq
\int_{\mathcal T} \nabla \overline U\nabla \chi dx -\omega \int_{\mathcal T} \overline U \chi dx = 3 \Lambda \int_{\mathcal T}  |U|^4\overline U \chi dx,
\eeq
where $\chi$ is a test function and $\Lambda$ the Lagrange multiplier. Let us assume that $U$ is the minimizer
of the variational problem (\ref{L6norm}). Testing $U$ with $\chi=U$ yields
\beq
\Lambda=\frac{B_{\omega}(U)}{3\mathcal{M}^6(\omega)}
\eeq
By setting $V := \Lambda^{\frac{1}{4}} U$, we obtain
\beq
\int_{\mathcal T} \nabla \overline V\nabla \chi dx -\omega \int_{\mathcal T} \overline V \chi dx = 3 \int_{\mathcal T}  |V|^4\overline V \chi dx,
\eeq
which is Euler--Lagrange equation for the variational problem \eqref{Nehari}) in the weak form.
Hence $V$ is a critical point of the variational problem (\ref{Nehari}).
It follows that $\Lambda \geq \frac{C_-}{3 C^2 \mathcal{M}^4(\omega)} = 1$.
Similarly, if $V$ is the minimizer of the variational problem \eqref{Nehari}, then
$U := \frac{V}{\Lambda^{\frac{1}{4}}}$ with $\Lambda := \frac{\| V \|^4_{L^6(\mathcal{T})}}{\mathcal{M}^4(\omega)} \geq 1$
is a critical point of the variational problem \eqref{L6norm}.
\end{proof}

\begin{remark}
The relation between the variational problems \eqref{Nehari} and \eqref{L6norm} in Proposition \ref{prop-Appendix-B}
does not allow to conclude that the minimizers of the two problems coincide. Notice however that if the minimizers
satisfy the same monotonicity properties stated in Theorem \ref{theorem-existence}, then
the conclusions of Theorem \ref{theorem-degeneracy} hold true and hence the minimizer of one problem
is at least a local minimum of the other problem.
\end{remark}

\begin{remark}
Thanks to the scaling transformation, the constrained minimization problem (\ref{L6norm})
can be normalized to the form (\ref{infB}) in the sense that the minimizers of both (\ref{infB}) and (\ref{L6norm})
are constant proportional to each other and the Euler--Lagrange equation for each problem is given by
the stationary NLS equation (\ref{eqstaz}).
\end{remark}

\section{Asymptotic computation of the integral (\ref{int-mass})}
\label{app-asymptotics}

Here we justify the asymptotic computation of the integral (\ref{int-mass}).
By using the exact solution (\ref{exact-solution}) and the asymptotic expansions
(\ref{expansions-rho}) and (\ref{expansion-k}) as $a \to \infty$, we obtain
\begin{eqnarray}
\label{int-1}
\mathcal{I}_{\varepsilon} := \int_0^{\pi \varepsilon^2} \rho(z) dz =
\int_0^{\pi \varepsilon^2 (1 + \mathcal{O}(e^{-2a}))} \frac{\left[ 1 + \mathcal{O}(e^{-2a}) \right] {\rm dn}^2(z;k)}{
2 - {\rm dn}^2(z;k) + \mathcal{O}(e^{-2 a})} dz,
\end{eqnarray}
where $\mathcal{O}(e^{-2a})$ stands for the error terms uniformly on the integration interval.
By using the asymptotic expansion 16.15 in \cite{AS} (justified in Proposition 4.6 and Appendix A in \cite{MarPel}),
we have
\begin{eqnarray}
\nonumber
{\rm dn}(z;k) & = & {\rm sech}(z) + \frac{1}{4} (1-k^2) \left[ \sinh(z) \cosh(z) + z \right]
\tanh(z) {\rm sech}(z) \\
\label{expansion-dn-extended}
& \phantom{t} & + \mathcal{O}((1-k^2)^2 z \cosh(z)),
\end{eqnarray}
for every $z \in [0,\pi \varepsilon^2]$ as long as $1-k^2 = \mathcal{O}(e^{-2 \pi \varepsilon^2})$
as in (\ref{k-unique}). Also recall that $\mathcal{O}(e^{-2a}) = \mathcal{O}(e^{-2\pi \varepsilon^2})$
as in (\ref{choices-unique}). Since the following integral converges as
$$
\int_0^{\infty} \frac{{\rm sech}(z)^2}{2 - {\rm sech}(z)^2} dz = \int_0^{\infty} \frac{dz}{\cosh(2z)} = \frac{\pi}{4},
$$
it follows that $\mathcal{I}_{\varepsilon}$ in (\ref{int-1}) can be expanded as $\varepsilon \to \infty$ in the form:
\begin{eqnarray}
\nonumber
\mathcal{I}_{\varepsilon} & = & \int_0^{\pi \varepsilon^2} \frac{{\rm dn}^2(z;k)}{
2 - {\rm dn}^2(z;k)} dz + \mathcal{O}(e^{-2 \pi \varepsilon^2}) \\
& = & \int_0^{\pi \varepsilon^2} \frac{\cosh(z)^2 {\rm dn}^2(z;k)}{
\cosh(2z)} dz + \mathcal{O}(e^{-2 \pi \varepsilon^2}).
\label{int-2}
\end{eqnarray}
Substituting (\ref{expansion-dn-extended}) into (\ref{int-2}) and moving terms of the order of
$\mathcal{O}(e^{-2\pi \varepsilon^2})$ from the integral to the remainder term yield
\begin{eqnarray*}
\mathcal{I}_{\varepsilon} & = & \int_0^{\pi \varepsilon^2} \frac{dz}{\cosh(2z)}
\left[ 1 + \frac{1}{4} (1-k^2) \left[ \sinh(z) \cosh(z) + z \right]
\tanh(z) + \mathcal{O}((1-k^2)^2 z \cosh(z)^2) \right]^2 \\
& \phantom{t} & \phantom{texttexttext} + \mathcal{O}(e^{-2 \pi \varepsilon^2}) \\
& = & \int_0^{\pi \varepsilon^2} \frac{dz}{\cosh(2z)}
\left[ 1 + \frac{1}{4} (1-k^2) \sinh(2z) \tanh(z)\right]
 + \mathcal{O}(e^{-2 \pi \varepsilon^2}) \\
& = & \frac{\pi}{4} + \frac{1}{4} (1 - k^2) \pi \varepsilon^2 + \mathcal{O}(e^{-2\pi \varepsilon^2}),
\end{eqnarray*}
where we have used the asymptotic balance of $1-k^2 = \mathcal{O}(e^{-2 \pi \varepsilon^2})$.
Substituting (\ref{k-unique}) into the latter expansion yields (\ref{int-mass}).

\subsection*{Acknowledgments.}

The present project was initiated when the first author visited Jeremy Marzuola at UNC at Chapel Hill.
Both the authors are very grateful to Jeremy for many useful discussions.
The authors also thank Sergio Rolando for valuable comments and Adilbek Kairzhan for help
in the proof of Lemma \ref{lemma-mass-monotonicity}. The authors appreciated the critical remarks of an anonymous referee, which pushed them to prove a stronger version of Theorem \ref{theorem-persistence}. 

D.Noja has received funding for this project from the European Union's Horizon 2020 research and innovation programme under the Marie Sk\l odowska-Curie grant no 778010 IPaDEGAN. D.E. Pelinovsky
acknowledges the support of the NSERC Discovery grant.

\end{document}